\documentclass{pspum-l}
\usepackage{graphicx}
\usepackage{amsmath,amsthm,amscd,amssymb}

\newtheorem{theorem}{Theorem}[section]
\newtheorem{lemma}[theorem]{Lemma}
\newtheorem{proposition}[theorem]{Proposition}
\newtheorem{corollary}[theorem]{Corollary}

\theoremstyle{definition}
\newtheorem{definition}[theorem]{Definition}

\theoremstyle{remark}
\newtheorem{remark}[theorem]{Remark}

\numberwithin{equation}{section}

\newcommand{\id}{I}

\newcommand{\e}{\hbox{\rm e}}
\newcommand{\bp}{{\mathbf{p}}}
\newcommand{\bw}{{\mathbf{w}}}

\newcommand{\bbR}{{\mathbb{R}}}
\newcommand{\R}{{\mathbb{R}}}

\newcommand{\bbZ}{{\mathbb{Z}}}
\newcommand{\Z}{{\mathbb{Z}}}
\newcommand{\C}{{\mathbb{C}}}

\newcommand{\bbC}{{\mathbb{C}}}

\newcommand{\cB}{{\mathcal B}}

\newcommand{\cD}{{\mathcal D}}

\newcommand{\cH}{{\mathcal H}}

\newcommand{\cM}{{\mathcal M}}

\newcommand{\cP}{{\mathcal P}}

\newcommand{\cT}{{\mathcal T}}

\newcommand{\cV}{{\mathcal V}}

\newcommand{\cX}{{\mathcal X}}
\newcommand{\cY}{{\mathcal Y}}

\newcommand{\lb}{\label}

\newcommand{\bv}{{\mathbf v}}

\newcommand{\by}{{\mathbf y}}

\newcommand{\ran}{\text{\rm{ran}}}

\newcommand{\bi}{\bibitem}
\newcommand{\sgn}{\text{\rm{sign }}}

\newcommand{\train}{\operatorname{Train}}
\newcommand{\dist}{\operatorname{dist}}

\newcommand{\mi}{\operatorname{Mas}}
\newcommand{\mo}{\operatorname{Mor}}

\numberwithin{equation}{section}

\renewcommand{\det}{\operatorname{det}}
\newcommand{\loc}{\operatorname{loc}}
\newcommand{\dom}{\operatorname{dom}}
\newcommand{\codim}{\operatorname{codim}}

\newcommand{\tr}{\operatorname{tr}}

\newcommand{\Sp}{\operatorname{Spec}}

\renewcommand{\Re}{\operatorname{Re }}
\renewcommand\Im{\operatorname{Im}}
\renewcommand{\ker}{\operatorname{ker}}

\begin{document}
\allowdisplaybreaks

\title{The Morse and Maslov indices for matrix Hill's equations}

\author{Christopher K.\ R.\ T. Jones}
\address{Mathematics Department,
The University of North Carolina at Chapel Hill Chapel Hill, NC 27599, USA}
\email{ckrtj@email.unc.edu}
\author{Yuri Latushkin}
\address{Department of Mathematics,
The University of Missouri, Columbia, MO 65211, USA}
\email{latushkiny@missouri.edu}
\author{Robert Marangell}
\address{Department of Mathematics and Statistics, The University of Sydney
Sydney, NSW 2006, Australia}
\email{r.marangell@maths.usyd.edu.au}
\thanks{Partially supported by the grants NSF DMS-0754705, DMS-1067929, DMS-0410267 and ONR N00014-05-1-0791, and by the Research Council and the Research
Board of the University of Missouri. {\bf This copy contains corrections in formulas (1.5), (3.6) of the printed version}}

\subjclass{Primary 53D12, 34L40; Secondary 37J25, 70H12}
\date{\today}
\keywords{Schr\"odinger equation, Hamiltonian systems, periodic potentials, eigenvalues, stability, differential operators, discrete spectrum}
\dedicatory{To Fritz Gesztesy on the occasion of his 60-th birthday with best wishes}
\begin{abstract}
 For Hill's equations with matrix valued periodic potential, we discuss relations between the Morse index, counting the number of
 unstable eigenvalues, and the Maslov index, counting the number of signed intersections of a path in the space of Lagrangian planes with a fixed plane. We adapt to the one dimensional periodic setting the strategy of a recent paper by J. Deng and C. Jones relating the Morse and Maslov indices for multidimensional elliptic eigenvalue problems. 
\end{abstract}

\maketitle

 \section{Introduction} \lb{s1}

Various results on Hill's equation are among many fundamental contributions made by Fritz Gesztesy in mathematical physics and analysis, see, for example, 
 \cite{GW96,GT09}.
 In the current  paper, we discuss a symplectic approach to counting positive $\theta$-eigenvalues for Hill's equations with matrix valued periodic potentials, that is, the values of $\lambda$ for which there exists a nontrivial solution of the eigenvalue problem
\begin{equation}\label{eq:hill}\begin{split}
Hy &:= y'' + V(x)y = \lambda y, \quad 
 y = (y_1(x), \ldots , y_n(x))^\top,
\end{split}
\end{equation}
that satisfies the boundary conditions
\begin{equation}\label{thetaBC}
y(L) = e^{i \theta} y(-L), \quad y'(L) = e^{i \theta} y'(-L).
\end{equation}
Here,   $x\in\R$, $y_i: \R \to \C$, $\theta\in[0,2\pi]$, and  $V(x)$ is an $n \times n$ symmetric matrix whose entires are real valued piecewise continuous periodic functions of period $2L$. We will denote by $H_\theta$ the differential operator in $L^2([-L,L])$ associated with the eigenvalue problem \eqref{eq:hill}, \eqref{thetaBC}.

A great deal of attention is devoted to Schr\"odinger operators with periodic potentials, see, e.g., \cite{MW,RS78,K97} and the bibliography therein. In the current paper, our main concern is the {\em Morse index}, $\mo(H_\theta)$, a ubiquitous number that appears in many areas, from variational calculus \cite{B56,D76,M63} to stability of traveling waves \cite{J88,SS08}, and which is defined as the dimension of the spectral subspace of a self-ajoint operator  corresponding to its positive (unstable) discrete eigenvalues. We will relate it to the {\em Maslov index}, $\mi(\Gamma,\cX)$, which is defined as the signed number of intersections of a curve $\Gamma$ in the set of Lagrangian planes with a given subvariety, called the {\em train} of a fixed Lagrangian plane $\cX$ (see \cite{arnold67,Arn85, J88, rs93, RoSa95} as well as more recent papers \cite{F, FJN, O90} and the bibliography therein for a discussion of this beautiful subject). 

One of the main motivations for studying the Maslov index in the context of second order differential operators was a generalization in \cite{Arn85} (for the case of matrix valued potentials) of the classical Sturm oscillation theorems; in connection with the latter we mention \cite{GST96}, \cite[Chapter 3]{G07} and the bibliography therein.
That the Morse and Maslov indices for periodic problems are related is of course well known (see, e.g., the classical sources 
\cite{D76,CZ84}, an excellent book \cite{A01} which has a detailed bibliography, and the important recent work done in \cite{CDB06, CDB09, CDB11}). However, all literature that we know deals only with the case of periodic eigenvalues corresponding to the particular case of $\theta=0$ or $\theta=2\pi$ (but also see \cite{SB}). 

More importantly, in the present paper we use a novel approach of determining the Maslov index borrowed from a recent paper \cite{DJ11} where the relations between the Morse and Maslov indices have been established in the {\em multidimensional} situation, in particular, for elliptic problems in a star-shaped domain $\cD$ in $\R^d$ containing zero. The main idea in \cite{DJ11} is to consider a family of ``shrinking'' domains $\cD_s$ parametrized by $s\in(0,1]$ and such that a point
$\mathbf{x}\in\partial\cD$ if and only if $s\mathbf{x}\in\partial\cD_s$.
Rescaling the original elliptic equation for $\lambda$-eigenfunctions from $\cD_s$ to $\cD$, one then defines a trace map 
$\phi^\lambda_s$ acting from the Sobolev space $H^1(\cD)$ into the trace space $H^{1/2}(\partial\cD)\times H^{-1/2}(\partial\cD)$. It maps a weak solution of the eigenvalue equation with no boundary conditions at all into a vector function on the boundary whose components are the Dirichlet and Neumann traces of the solution. Using Green's formulas, one defines a symplectic structure in the trace space so that if $\cY_{s,\lambda}$
denotes the set of the weak solutions then $\phi_s^\lambda(\cY_{s,\lambda})$ forms a curve in the set of Fredholm Lagrangian planes. The boundary conditions define a plane, and an intersection of the curve with the train of the plane defined via the boundary conditions corresponds to an eigenvalue of the elliptic operator at hand, eventually leading to a formula relating the Morse and Maslov indices. 

In the current paper, for the boundary value problem \eqref{eq:hill}, \eqref{thetaBC} on $[-L,L]$, following the strategy in \cite{DJ11}, we consider a family, parametrized by $s\in(0,L]$, of boundary value problems for \eqref{eq:hill} on the segments $[-s,s]$ with the boundary conditions  
\begin{equation}\label{sthetaBC}
y(s) = e^{i \theta} y(-s),\quad  y'(s) = e^{i \theta} y'(-s).
\end{equation} 
By changing  $s$ and $\lambda$ and using the traces of solutions of the differential equation \eqref{eq:hill} at the boundary of the segment $[-s,s]$, we construct a path in the set of finite dimensional Lagrangian planes. The construction of the path is the first crucial  ingredient of the current paper. The second key point is to utilize and further develop an idea from \cite{Ga93} to augment the first order system corresponding to \eqref{eq:hill} by considering a supplementary linear complex $(2n\times 2n)$ first order ODE system with the coefficient $\frac{i\theta}{2s}\,I_{2n}$ whose solutions automatically satisfy the boundary conditions \eqref{sthetaBC}. This allows one to replace the $e^{i\theta}$-periodic boundary conditions in \eqref{thetaBC}, \eqref{sthetaBC} by certain ``Dirichlet-type'' boundary conditions for the augmented system.

Our plan then is to re-write the eigenvalue equation \eqref{eq:hill} as a complex $(2n\times 2n)$ first order system, separate the real and imaginary parts of the solutions in the eigenvalue equation and the boundary conditions, thus arriving at a $(4n\times 4n)$ real system, consider the augmented $(8n\times 8n)$ real system, and then to define a trace map, $\Phi_s^\lambda$, for each $s\in(0,L]$ and $\lambda\in\R$, that maps a solution $(\bp,\bw)^\top$ of the augmented system on $[-L,L]$ with no boundary conditions at all into its trace $(\bp(-s),\bw(-s),\bp(s),\bw(s))^\top\in\R^{16n}$ on the boundary of of the segment $[-s,s]$.  This leads to the critical observation (see Proposition \ref{cor:mult} below) that if $Y_{s,\lambda}$ denotes the set of the solutions of the augmented  system then $\lambda$ is an eigenvalue of \eqref{eq:hill}, \eqref{sthetaBC} on $[-s,s]$ if and only if the plane $\Phi_s^\lambda(Y_{s,\lambda})$ intersects the plane $X\times X$ in $\R^{16}$ consisting of vectors whose 
respective $\bp(\pm s)$- and $\bw(\pm s)$-components are equal; here and below we denote $\bp=(p,q)^\top$, $\bw=(w,z)^\top$, and use notation
\begin{equation}\label{DefX1}
X = \{ (p,q,w,z)^\top\in \R^{8n}\,\big|\, p=w, q = z \}.
\end{equation} 
Thus, the ``Dirichlet-type'' boundary condition $\Phi_s^\lambda\big((\bp,\bw)^\top\big)\in X\times X$ replaces the $e^{i\theta}$-periodic boundary condition \eqref{sthetaBC}.

There is a natural symplectic structure in $\R^{16n}$ such that the planes $\Phi_s^\lambda(Y_{s,\lambda})$ and $X\times X$ in $\R^{16n}$ are Lagrangian (see Theorem \ref{th:lagrange}). Thus, one can consider crossings with the train of $X\times X$ of the Lagrangian curve $\Gamma=\Gamma_1\cup\Gamma_2\cup\Gamma_3\cup\Gamma_4$ formed by  $\Phi_s^\lambda(Y_{s,\lambda})$ when $(\lambda,s)$ runs along the boundary of the square $[s_0,L]\times[0,\lambda_\infty]$, for a small $s_0>0$ and a large $\lambda_\infty$, where $\Gamma_j$, $j=1,2,3,4$, correspond to the four sides of the square, see Figure \ref{F1}.
We stress that $\Gamma$ depends on the choice of $s_0$ and $\lambda_\infty$ while the location of the crossings of course depends on $\theta$; we sometimes write $\Gamma_{(\theta,s_0)}$ and $\Gamma_{j, (\theta,s_0)}$.

 \begin{figure}
\begin{picture}(100,100)(-20,0)
\put(2,2){0}
\put(10,5){\vector(0,1){95}}
\put(5,10){\vector(1,0){95}}
\put(40,20){\vector(1,0){5}}
\put(30,80){\vector(-1,0){5}}
\put(80,50){\vector(0,1){5}}
\put(10,60){\vector(0,-1){5}}
\put(71,40){\text{\tiny $\Gamma_2$}}
\put(12,40){\text{\tiny $\Gamma_4$}}
\put(45,73){\text{\tiny $\Gamma_3$}}
\put(45,13){\text{\tiny $\Gamma_1$}}
\put(100,12){$\lambda$}
\put(12,100){$s$}
\put(80,20){\line(0,1){60}}
\put(80,8){\line(0,1){4}}
\put(78,0){$\lambda_\infty$}
\put(10,20){\line(1,0){70}}
\put(10,80){\line(1,0){70}}
\put(0,78){$L$}
\put(0,18){$s_0$}
\put(10,30){\circle*{4}}
\put(10,50){\circle*{4}}
\put(10,80){\circle*{4}}
\put(40,80){\circle*{4}}
\put(60,80){\circle*{4}}
\put(20,87){{\tiny \text{$\theta$-eigenvalues}}}
\put(11,25){{\tiny \text{no $(\theta,s_0)$-eigenvalues}}}
\put(-10,20){\rotatebox{90}{{\tiny conjugate points}}}
\put(90,20){\rotatebox{90}{{\tiny no conjugate points}}}
\end{picture}
\caption{ $\lambda=0$ is a $\theta$-eigenvalue, $\theta\in(0,2\pi)$, and $s_0$ is small enough}\label{F1}
\end{figure}
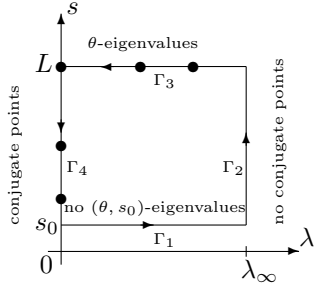 

A homotopy argument implies that the Maslov index $\mi(\Gamma, X\times X)$ of the entire curve $\Gamma$ is equal to zero (see Corollary \ref{cor3.9}). By general properties of the Maslov index one infers $\mi(\Gamma, X\times X)=\sum_{j=1}^4\mi(\Gamma_j, X\times X)$. For $\theta\in(0,2\pi)$ one can show that there are no crossings along $\Gamma_1$ (when $s=s_0$ and $\lambda\in[0,\lambda_\infty]$) provided $s_0$ is chosen small enough (see Lemma \ref{lem:llss}).
For $\theta=0$ or $\theta=2\pi$, assuming that the potential $V$ is continuous at the point $x=0$, and $s_0>0$ is small enough, one can show that the number of crossings along $\Gamma_1$ is equal to the number $\mo(V(0))$ of positive eigenvalues of the matrix $V(0)$ (Lemma \ref{llssnew}).
Since the spectrum of the operator $H_\theta$ is bounded from above, there are no crossings along $\Gamma_2$  (when $\lambda=\lambda_\infty$ and $s\in[s_0,L]$) provided $\lambda_\infty$ is chosen large enough, see Lemma \ref{lem:llss}.  

The crossings of the  curve $\Gamma_3$ (when $s=L$ and $\lambda\in[0,\lambda_\infty]$), correspond to the $\theta$-eigenvalues of  \eqref{eq:hill}, \eqref{thetaBC}. A local computation shows that all crossings along $\Gamma_1$ have the same signs and all crossings along $\Gamma_3$ have the same signs, see Lemma \ref{lem:cross}. This important monotonicity property of the Maslov index implies that the Morse index $\mo(H_\theta)$ is equal to the number of crossings along $\Gamma_3$ counting their multiplicities. 

It turns out that the crossings along $\Gamma_4$ (when $\lambda=0$ and $s\in[s_0,L]$) correspond to conjugate points of the Hill's equation on $[-L,L]$, that is, to the points $s$ where the number $e^{i\theta}$ is an eigenvalue of the propagator of this equation transforming the value of its solution at the point $-s$ into the value at the point $s$ (Proposition \ref{cor:mult}). Thus, $\mi(\Gamma_4, X\times X)$ can be viewed as the Maslov index of  the boundary value problem \eqref{eq:hill}, \eqref{thetaBC} for the Hill equation.
 Yet another local computation shows that all crossings along $\Gamma_4$ have the same sign provided that, in addition, the potential is sign definite (see Lemma \ref{lem:scross}).

 Since $\mi(\Gamma, X\times X)=0$, we therefore arrive at the 
 desired formula
 \begin{equation}\label{MoMiform}
 2\mo_\bbC(H_\theta)=
 \begin{cases}-\mi_\bbR(\Gamma_4, X\times X), & \text{if  $\theta\in(0,2\pi)$, }\\ & \text{for small $s_0=s_0(\theta)>0$,}\\
 -\mi_\bbR(\Gamma_4, X\times X)+\mo_\bbR(V(0)), & \text{if $\theta=0$ or $\theta=2\pi$,}\\ &\text{for small $s_0>0$},
 \end{cases}
 \end{equation}
 relating the Maslov index of the boundary value problem for the Hill equation and the Morse index of the corresponding differential operator (see Theorem \ref{th:mas2} summarizing our results).
 For instance, for a fixed $s_0>0$, when $\theta$ changes from a positive value to zero, the crossings move from $\Gamma_4$ to $\Gamma_1$ through the left bottom corner of the square in Figure \ref{F1}, thus keeping the proper balance in formula \eqref{MoMiform}.

The paper is organized as follows. In Section \ref{heas} we set up the stage and introduce the augmented system for the Hill equation \eqref{eq:hill}. After a brief reminder of basics on the Maslov index, in Section \ref{secsa} we introduce an appropriate Lagrangian structure, and relate the crossings of the path $\Phi_s^\lambda$ to the eigenvalues of differential operators. In Section \ref{mmi} we prove monotonicity of the crossings, and summarize the main results of the paper.
Finally, in Section \ref{sec:me} we conduct several numerical experiments calculating the Maslov and Morse indices for a particular Mathieu equation.

 {\bf Notations.} We denote by $I_n$ and $0_n$ the $n\times n$ identity and zero matrix.  For an $n\times m$ matrix $A=(a_{ij})_{i=1,j=1}^{n,m}$
 and a $k\times\ell$ matrix $B=(b_{ij})_{i=1,j=1}^{k,\ell}$, we denote by
 $A\otimes B$ the Kronecker product, that is, the $nk\times m\ell$ matrix composed of $k\times\ell$ blocks $a_{ij}B$, $i=1,\dots n$, $j=1,\dots m$. We let $\langle\cdot\,,\cdot\rangle_{\R^n}$ denote the real scalar product in the space $\bbR^n$ of $n\times 1$ vectors, and let $\top$ denote transposition. We denote by $A \oplus B$ the matrix $\begin{pmatrix} A & 0 \\ 0 & B \end{pmatrix}$ and use notation $J = \begin{pmatrix} 0 & 1 \\ -1 & 0 \end{pmatrix}$ for the standard symplectic matrix. When $a=(a_i)_{i=1}^n\in\bbR^n$ and $b=(b_j)_{j=1}^m\in\bbR^m$ are $(n\times 1)$ and $(m\times 1)$ column vectors,  we use notation $(a,b)^\top$ for the $(n+m)\times 1$ column vector with the entries $a_1,\dots,a_n,b_1,\dots,b_m$ (just avoiding  the use of $(a^\top,b^\top)^\top$). We denote by $\cB(\cX)$ the set of linear bounded operators on a Hilbert space $\cX$ and by $\Sp(T)=\Sp(T; \cX)$ the spectrum of an operator on $\cX$.
 
 {\bf Acknowledgment.} We thank Konstantin Makarov and Holger Dullin for their valuable suggestions.
 
 \section{Hill's equation and an augmented equation}\label{heas}
 
We start with the eigenvalue problem \eqref{eq:hill}, where 
 we consider $\lambda \in \R$, and consider complex valued solutions to (\ref{eq:hill}). Setting 
\begin{equation}\label{Defpq}\begin{split}
p_i &:= (\Re(y_i), \Im(y_i))^\top \in \R^{2}, \quad p:= (p_1, \ldots, p_n)^\top \in \R^{2n},\\
q_i & := (\Re(y_i'), \Im(y_i'))^\top \in \R^{2}, \quad q: = (q_1, \ldots, q_n)^\top \in \R^{2n},\end{split}\end{equation}
we can write (\ref{eq:hill}) as follows:
\begin{equation}\label{eq:hillsys}
\begin{pmatrix} p \\ q \end{pmatrix}' = \begin{pmatrix} 0_{2n} & I_{2n} \\ \lambda I_{2n} -  V(x) \otimes I_2 & 0_{2n} \end{pmatrix} \begin{pmatrix} p \\ q \end{pmatrix}.
\end{equation}
It is sometimes convenient to denote ${\bf p} := (p,q)^\top \in \R^{4n}$ and to write \eqref{eq:hillsys} as
\begin{equation} \label{eq:hillp}
{\bf p}' = A(x,\lambda) {\bf p}, \quad A(x,\lambda)= \begin{pmatrix} 0_{2n} & I_{2n} \\ \lambda I_{2n} -  V(x) \otimes I_2 & 0_{2n} \end{pmatrix}.
\end{equation} 

We are interested in studying bounded on $\bbR$ solutions of \eqref{eq:hill}. To this end, for each $\theta \in [0,2\pi]$, we will examine for which $\lambda$ there exists a nontrivial solution $y$ of \eqref{eq:hill} that satisfies the boundary condition \eqref{thetaBC}.
In particular, if $\theta=0$ or $\theta=2\pi$ we have periodic boundary conditions, and if $\theta=\pi$ we have antiperiodic ones. Equivalently, using \eqref{Defpq} and  writing out \eqref{thetaBC} in real and imaginary parts, we seek a nontrivial solution
$\bp=(p, q)^\top$ of  \eqref{eq:hillsys} such that the following  boundary condition is satisfied:
\begin{equation}\label{eq:pereigen}
\begin{pmatrix} p(L) \\ q(L) \end{pmatrix} 
= \begin{pmatrix} I_n \otimes U(\theta)  & 0 \\ 0 & I_n \otimes U(\theta) \end{pmatrix}\begin{pmatrix} p(-L) \\ q(-L) \end{pmatrix}, 
\end{equation}
where we denote 
\begin{equation}\lb{defUoftheta}
U(\theta):= \begin{pmatrix} \cos \theta & -\sin \theta \\ \sin \theta & \cos \theta \end{pmatrix}.
\end{equation}
In the notation of equation (\ref{eq:hillp}), condition (\ref{eq:pereigen}) is written as
\begin{equation}\label{eq:eigenshort}
{\bf p} (L) = (I_{2n} \otimes U(\theta)) {\bf p} (-L).
\end{equation}
Since the boundary conditions \eqref{thetaBC} and \eqref{sthetaBC} are the same in the case when $\theta=0$ or $\theta=2\pi$, out of these two possibilities we will always consider only the former.

We now briefly discuss the spectrum of the operators associated with \eqref{eq:hill}. On the space $L^2(\R)$ of $(n\times 1)$ complex vector valued functions, or on the space  $BUC(\R)$ of bounded uniformly continuous complex vector valued functions, one can associate to equation \eqref{eq:hill} a differential operator, $H$, defined by $H=\frac{d^2}{dx^2}+V(x)$, whose domain is given by the following formula (we recall that the potential $V$ is bounded):
\begin{equation}\label{domH}
\dom(H)=\Big\{ y\in L^2(\R)\Big|\, y,y'\in AC_{\rm loc}(\R), y''\in L^2(\R)\Big\}
\end{equation}
(for the space $BUC(\R)$ one has to replace the space $L^2(\R)$ in \eqref{domH} by $BUC(\R)$). There is a standard way, see \cite[Section XIII.16]{RS78}, of associating with $H$ a family of operators, $H_\theta$, with $\theta\in[0,2\pi]$, acting in $L^2([-L,L])$ and induced by the complex  boundary conditions \eqref{thetaBC}. Indeed, we may identify $L^2(\bbR)$ and $L^2([0,2\pi]; L^2([-L,L]))=L^2([0,2\pi]\times[-L,L])$ by introducing,
see \cite[eqn. (147)]{RS78}, a family of operators $W_\theta: L^2(\R)\to L^2([-L,L])$ by
\begin{equation}\lb{defUtheta}
(W_\theta y)(x)=\sum_{n\in\Z}\e^{-in\theta }y(x+2Ln),\quad x\in[-L,L].
\end{equation}
Obviously, $(W_\theta y)(L)=\e^{i\theta}(W_\theta y)(-L)$, and analogously  for the derivative $y'$, leading to the fact the $H$ is similar to the direct integral,
$\oplus\int_0^{2\pi} H_\theta\frac{d\theta}{2\pi}$, of the operators $H_\theta$ defined in $L^2([-L,L])$ as follows: $H_\theta=\frac{d^2}{dx^2}+V(x)$ with
\begin{equation}\lb{defLtheta}\begin{split}
 \dom(H_\theta)=\Big\{&
y\in L^2([-L,L])\Big|
 y,y'\in AC_{\rm loc}([-L,L]),\\ & y''\in L^2([-L,L])\text{ and the boundary condition \eqref{thetaBC} holds}\Big\}.
\end{split}\end{equation}
Similarly, one can introduce the operator $H_\theta$ on the space $C([-L,L])$ of continuous functions by replacing $L^2([-L,L])$ in \eqref{defLtheta} by $C([-L,L])$. 

For each $\theta\in[0,2\pi]$, the spectrum $\Sp(H_\theta)$ consists of discrete eigenvalues; when $\theta$ varies, they fill up spectral bands with or without spectral gaps between them, thus forming the spectrum $\Sp(H)$, see \cite{MW,RS78} for a detailed exposition. 
\begin{definition}
We say that $\lambda$ is a {\em $\theta$-eigenvalue} of equation (\ref{eq:hill}) if there is a nonzero solution of (\ref{eq:hillsys}) such that the boundary condition (\ref{eq:pereigen}) is satisfied. 
\end{definition}
For each $\lambda\in\R$, we let $\Psi_A(x,\lambda)$ denote the fundamental matrix solution to equation (\ref{eq:hillp}) such that $\Psi_A(-L,\lambda) = \id_{4n}$  and, for each $s\in(0,L]$, let  $M_A(s,\lambda):=\Psi_A(s,\lambda)\Psi_A(-s,\lambda)^{-1}$ denote its propagator so that $\bp(s)=M(s,\lambda)\bp(-s)$ for a solution of \eqref{eq:hillp}. 
In particular, $M_A(L,\lambda)=\Psi_A(L,\lambda)$ denotes the monodromy matrix for \eqref{eq:hillp}. We recall that in \cite{Ga93}, $\lambda$ is said to be a {\em $\gamma$-eigenvalue} if $\gamma\in\{\gamma\in\C: \, |\gamma|=1\}$ is an eigenvalue of the monodromy matrix of equation (\ref{eq:hillsys}). We note that our definition of $\theta$-eigenvalue is consistent with the definition of $\gamma$-eigenvalue, with $\gamma = e^{i \theta}$, given in \cite{Ga93}, as the following proposition shows. 

\begin{proposition} \label{prop:garequiv}
On $L^2(\R)$ or $BUC(\bbR)$, the following assertions are equivalent:
\begin{enumerate}
\item[(i)] \, $\lambda\in \Sp(H)$; 
\item[(ii)]\, equation \eqref{eq:hillsys} has a bounded solution on $\R$;
\item[(iii)]\, $\Sp(M_A(L,\lambda))\cap\{\gamma\in\C: |\gamma|=1\}\neq\emptyset$;
\item[(iv)] equation \eqref{eq:hillsys} has a solution on $[-L,L]$
satisfying \eqref{eq:pereigen} for a $\theta\in[0,2\pi]$;
\item[(v)] \,  $\lambda\in \Sp(H_\theta)$ for a $\theta\in[0,2\pi]$.
\end{enumerate}
\end{proposition}
\begin{proof} This follows immediately from Proposition 2.1 in \cite {Ga93} and its proof and from the results in \cite[Section XIII.16]{RS78}. The equivalence of the last three assertions is also proved in a slightly more general Proposition \ref{cor:mult} below.\end{proof}

We will now introduce a family of systems of equations parametrized by $s\in(0,L]$ which augment (\ref{eq:hillsys}). Each system will be a linear constant coefficient system whose solutions satisfy the same boundary condition (\ref{eq:pereigen}) as our original system but with $L$ replaced by $s$.  To this end let us consider the system
\begin{equation}\label{eq:const}
\begin{split}
\begin{pmatrix} \zeta \\ \xi \end{pmatrix} ' = \begin{pmatrix} \frac{i \theta}{2 s} & 0 \\ 0 & \frac{i \theta}{2s} \end{pmatrix} \begin{pmatrix} \zeta \\ \xi \end{pmatrix}, \quad \zeta, \xi:\R \to \C^n, \theta\in[0,2\pi], \, s\in(0,L].
\end{split}
\end{equation}
Setting 
\begin{equation}\begin{split}\label{Defwz}
w_i &:= (\Re(\zeta_i), \Im(\zeta_i))^\top \in \R^2, \quad w := (w_1, \ldots, w_n)^\top \in \R^{2n},\\ 
z_i & := (\Re(\xi_i), \Im(\xi_i))^\top \in \R^2, \quad z := (z_1, \ldots,z_n)^\top\in \R^{2n},\end{split}\end{equation} we observe that $w$ and $z$ satisfy the following system of ODEs: 
\begin{equation}\label{eq:rot}
\begin{pmatrix} w \\ z \end{pmatrix}' = \begin{pmatrix}  I_n \otimes \mathfrak{u}(s,\theta) & 0 \\ 0& I_n \otimes \mathfrak{u}(s,\theta) \end{pmatrix} \begin{pmatrix} w \\ z \end{pmatrix},\quad
 \mathfrak{u}(s,\theta) := \begin{pmatrix} 0 & \frac{-\theta}{2s} \\ \frac{\theta}{2s} & 0 \end{pmatrix}.
 \end{equation}
 Any solution of \eqref{eq:const}, respectively, \eqref{eq:rot} will automatically satisfy the same boundary conditions as in \eqref{thetaBC}, respectively, \eqref{eq:pereigen}, with $L$ replaced by $s$, that is, the boundary conditions $( \zeta(s), \xi(s))^\top = e^{i \theta} (\zeta(-s), \xi(-s))^\top$, respectively, 
\begin{equation}\label{eq:wpereigen}
\begin{pmatrix} w(s) \\ z(s) \end{pmatrix} 
= \begin{pmatrix} I_n \otimes U(\theta) & 0 \\ 0 & I_n \otimes U(\theta) \end{pmatrix}\begin{pmatrix} w(-s) \\ z(-s) \end{pmatrix}.
\end{equation}
As before, sometimes it is convenient to write equation (\ref{eq:rot}) in a more condensed form denoting ${\bf w} := (w, z)^\top$, and writing equation (\ref{eq:rot}) as the following equation with $x$-independent coefficient:
\begin{equation}\label{eq:rotshort}
{\bf w} ' = B(s,\theta){\bf w}, \quad B(s,\theta):=I_{2n} \otimes \mathfrak{u}(s,\theta).
\end{equation}
For each $s\in(0,L]$, we let $\Psi_B(x,s)$ denote the fundamental matrix solution to the equation (\ref{eq:rot}) such that $\Psi_B(-L,\lambda) = \id_{4n}$, and remark that 
\begin{equation}\label{defPsiB}
\Psi_B(x,s) = I_{2n} \otimes e^{\mathfrak{u}(s,\theta)(x+L)} = I_{2n} \otimes
\begin{pmatrix} 
\cos \frac{\theta}{2s}(x+L) & -\sin \frac{\theta}{2s}(x+L) \\ \sin \frac{\theta}{2s}(x+L) & \cos \frac{\theta}{2s}(x+L) 
 \end{pmatrix}
\end{equation}
is an orthogonal matrix: $\Psi_B(x,s)^\top=\big(\Psi_B(x,s)\big)^{-1}$.

It is sometimes convenient to combine (\ref{eq:hillp}) and (\ref{eq:rotshort}) as follows: 
\begin{align}\label{eq:augodeshort}
&\begin{pmatrix} {\bf p } \\ {\bf w} \end{pmatrix}'   = \begin{pmatrix} A (x, \lambda) & 0_{4n} \\ 0_{4n} & B(s,\theta) \end{pmatrix}
 \begin{pmatrix} {\bf p } \\ {\bf w} \end{pmatrix},\, x\in[-L,L],\,\theta\in[0,2\pi],\, s\in(0,L].
 \end{align}

We will now
reformulate the boundary value problems for equations \eqref{eq:hillsys} and \eqref{eq:rot} with $s=L$ in a way amenable for symplectic analysis.
We consider $X$ defined in \eqref{DefX1} as a $4n$-plane in $\R^{8n}$. We claim that $\lambda$ is a $\theta$-eigenvalue of equation (\ref{eq:hill}) if and only if there is a nonzero solution to the following (augmented) boundary value problem: 
\begin{align}\label{eq:Laugode} 
& \begin{pmatrix} p \\ q \\ w \\ z \end{pmatrix}'  = \begin{pmatrix} 
0 & I_{2n} & 0 & 0 \\ 
\lambda I_{2n} -I_2 \otimes V(x) & 0 & 0 & 0 \\
0 & 0 & I_n \otimes \mathfrak{u}(L,\theta) & 0 \\
0 & 0 & 0 & I_n \otimes \mathfrak{u}(L,\theta)
\end{pmatrix} 
\begin{pmatrix} p \\ q \\ w \\ z \end{pmatrix}, \\
\label{eq:Laugbdc}
& 
 \big(p(-L), q(-L), w(-L), z(-L)\big)^\top,
\big(p(L), q(L), w(L), z(L)\big)^\top \in X.
\end{align}
It is convenient to write (\ref{eq:Laugode}) and (\ref{eq:Laugbdc}) as follows: 
\begin{align}\label{eq:Laugodeshort}
&\begin{pmatrix} {\bf p } \\ {\bf w} \end{pmatrix}'   = \begin{pmatrix} A (x, \lambda) & 0_{4n} \\ 0_{4n} & B(L,\theta) \end{pmatrix}
 \begin{pmatrix} {\bf p } \\ {\bf w} \end{pmatrix},\\
\label{Laugbdcshort}
& 
\begin{pmatrix} {\bf p }(-L) \\ {\bf w}(-L) \end{pmatrix}, \begin{pmatrix} {\bf p }(L) \\ {\bf w}(L) \end{pmatrix} \in X.
\end{align} 
To justify the claim, we note that if $\bw$ satisfies \eqref{eq:rotshort} with $s=L$ then $\bw$ automatically satisfies \eqref{eq:wpereigen}
with $s=L$. Thus, if \eqref{Laugbdcshort} holds then $\bp$ satisfies \eqref{eq:eigenshort}. Conversely, given a $\bp$ satisfying \eqref{eq:eigenshort}, pick a solution $\bw$ of \eqref{eq:rotshort} 
with $s=L$ such that $\bw(-L)=\bp(-L)$. Then \eqref{Laugbdcshort} holds.

\section{A symplectic approach to counting eigenvalues}\label{secsa}
We begin by recalling some notions regarding symplectic structures and the Maslov index; for a detailed exposition see \cite{arnold67,Arn85,rs93,RoSa95} and a review \cite{F},
for a brief but extremely informative account see \cite{FJN}.

A skew-symmetric non-degenerate quadratic form $\omega$ on $\R^{2n}$ is called symplectic. Symplectic forms are in one-to-one correspondence with orthogonal skew-symmetric matrices $\Omega$, such that $\Omega^\top=\Omega^{-1}=-\Omega$, via the relation $\omega(v_1,v_2)=\langle  v_1, \Omega v_2\rangle_{\R^{2n}}$, $v_1, v_2\in\R^{2n}$. A real Lagrangian plane $V$ is an $n$-dimensional subspace in $\R^{2n}$ such that $\omega(v_1,v_2)=0$ for all $v_1,v_2\in V$. 
The set of all Lagrangian planes in $\R^{2n}$ is denoted by $\Lambda(n)$.

Let $\train(V)$ denote the train of a Lagrangian plane $V\in\Lambda(n)$, that is the set of all Lagrangian planes whose intersection with $V$ is non trivial. Obviously, $\train(V)=\cup_{k=1}^n\cT_k(V)$ where $\cT_k(V)=\big\{V_0\in\Lambda(n)\,\big|\, \dim(V\cap V_0)=k\big\}$. Each set $\cT_k(V)$ is an algebraic submanifold of $\Lambda(n)$ of codimension $k(k+1)/2$. In particular, $\codim\cT_1(V)=1$; moreover, $\cT_1(V)$ is two-sidedly imbedded in $\Lambda(n)$, that is, there is a continuous vector field tangent to $\Lambda(n)$ which is transversal to $\cT_1(V)$. Hence, one can speak about the positive and negative sides of $\cT_1(V)$. 
Thus, given a smooth closed curve $\Phi$ in $\Lambda(n)$ that intersects $\train V$ transversally (and thus in $\cT_1(V))$, one can define the {\em Maslov index} $\mi(\Phi,V)$ as the signed number of intersections.

We now recall a more detailed definition of the Maslov index as well as how to calculate it from local data. Let $\Phi:[a,b]\to\Lambda(n)$ be a smooth path. A {\em crossing} is a point $t_0\in(a,b)$ of intersection of the path $\{\Phi(t): t\in[a,b]\}$ with $\train(V)$.  Let $t_0\in(a,b)$ be a crossing for a smooth path $\Phi$, that is, assume that $\Phi(t_0)\cap V\neq\{0\}$. Let $V^\bot$ be a subspace in $\R^{2n}$ transversal to $\Phi(t_0)$. Then $V^\bot$ is transversal to $\Phi(t)$ for all $t\in[t_0-\varepsilon,t_0+\varepsilon]$ for $\varepsilon>0$ small enough.
Thus, there exists a smooth family of matrices, $\phi(t)$, viewed as operators from $\Phi(t_0)$ into $V^\bot$, so that $\Phi(t)$ is the graph of $\phi(t)$ for $|t-t_0|\le\varepsilon$. The bilinear form $Q_\cM=Q_\cM(\Phi(t_0),V)$ defined by
\begin{equation}\label{QMF}
Q_\cM({\rm v}, {\rm w})=\frac{d}{dt}\omega({\rm v}, \phi(t){\rm w})\big|_{t=t_0} \,\text{ for }\, {\rm v}, {\rm w}\in \Phi(t_0)\cap V,
\end{equation} is called the {\em crossing form}. 

A crossing is called {\em regular} if the crossing form is non degenerate. At a regular crossing $t_0$, denote the signature of the crossing form by $\sgn Q_\cM(\Phi(t_0),V)$.  The {Maslov index} $\mi(\Phi, V)$ of the path $\Phi$ with only regular crossings of $\train (V)$ is then defined as 
\begin{equation}\label{DefMI}\begin{split}
\mi(\Phi, V) &:= \frac{1}{2} \sgn Q_\cM(\Phi(a),V) \\ &\quad + \sum_{t \in (a,b)} \sgn Q_\cM(\Phi(t),V) + \frac{1}{2} \sgn Q_\cM(\Phi(b),V),
\end{split}
\end{equation}
\noindent where the summation above is over all crossings $t$ (one can verify that regular crossings are isolated \cite{rs93}). At the endpoints, take the appropriate left or right limit definition of the derivative in \eqref{QMF} to compute the bilinear form $Q_\cM$ (and hence its signature). We remark that now we have a Maslov index even if the crossing does not take place in $\cT_1$. It will sometimes be convenient to refer to the absolute value of the local Maslov index of a crossing as the {\em multiplicity} of the crossing. In the sequel, a curve with only regular crossings will also be called {\em regular}. From the context it should always be clear whether regular refers to a crossing or to the curve itself. 

The important features of the Maslov index for this work are summarized below.
\begin{theorem} \cite{rs93}
\begin{enumerate}
\item \label{it:nat} {\em (Naturality)} If $T$ is a symplectic linear transformation then 
$$ \mi(T \Phi(t), T V) = \mi(\Phi, V). $$
\item \label{it:cat} {\em (Catenation)} For $a < c < b $ 
$$\mi(\Phi, V)  = \mi(\Phi(\cdot)\big|_{[a,c]},V) +  \mi(\Phi(\cdot)\big|_{[c,b]},V).$$
\item \label{it:hom} {\em (Homotopy)} Two paths $\Phi_0, \Phi_1 :[a,b] \to \Lambda(n)$, with $\Phi_0(a) = \Phi_1(a)$ and $\Phi_0(b) = \Phi_1(b)$, are homotopic with fixed endpoints if and only if they have the same Maslov index. 
\end{enumerate}
\end{theorem}

\begin{remark}[The generic case] \label{rem:generic}
A crossing $t_0$ is called {\em simple} if it is regular and $\Phi(t_0) \in \cT_1(V)$. A curve has only simple crossings if and only if  it is transverse to every $\cT_k(V)$.  Suppose that a curve $\Phi:[a,b] \to \Lambda(n)$ with $\Phi(a), \Phi(b) \in \cT_0(V) :=\big\{V_0\in\Lambda(n)\,\big|\, \dim(V\cap V_0)=0\big\}$ has only simple crossings. Then the two-sidedness of $\cT_1(V)$ allows one to define $m_+$ to be the number of crossings by which $\Phi(t)$ passes from the negative side of $\cT_1(V)$ to the positive side, and $m_-$ to be the number of crossings from negative to positive. We then have that $\mi(\Phi,V) = m_+ - m_-$. 
 \end{remark}
 \begin{remark} \label{rem:multiplicity}
We remark that at a regular crossing $t_0$ the Maslov index of the path $\Phi : [t_0-\varepsilon, t_0 + \varepsilon ] \to \Lambda(n) $, for small enough $\varepsilon$, is equal to the signature of the crossing form at the crossing. In particular, the crossing is called {\em positive} (respectively {\em negative}) if the crossing form is positive (negative) definite. In this case the local Maslov index at the crossing is equal to plus (respectively minus) the dimension of the subspace $\Phi(t_0) \cap V$ (i.e. the multiplicity of the crossing is the real dimension of this subspace). 
 \end{remark}
We will now return to the augmented system \eqref{eq:augodeshort}.
Following \cite{DJ11}, for each $\lambda\in\R$ and $s\in(0,L]$ we now define the following set of vector valued functions on $[-L,L]$:
\begin{equation}\begin{split}
Y_{s,\lambda} =\Big\{& ({\bf p } , {\bf w} )^\top  \Big|\, {\bf p } , {\bf w} \in AC_{\rm loc}([-L,L]),\\ & \quad \text{and $({\bf p } , {\bf w})^\top$ is a  solution of  (\ref{eq:augodeshort}) on $[-L,L]$} \Big\}.\end{split}\label{defY}
\end{equation}
That is, we consider the ($8n$ dimensional) solution space to the augmented equation \eqref{eq:augodeshort}, defined on $[-L,L]$, without any boundary conditions at all. We stress that by solutions $(\bp,\bw)^\top$ of \eqref{eq:augodeshort} on $[-L,L]$ we understand the mild solutions, that is, absolutely continuous vector valued functions such that \eqref{eq:augodeshort} holds for almost all $x\in[-L,L]$; in other words,
$\bp(x)=\Psi_A(x,\lambda)\bp(-L)$ and $\bw(x)=\Psi_B(x,s)\bw(-L)$, $x\in[-L,L]$, where $\Psi_A(\cdot,\lambda)$ and $\Psi_B(\cdot,s)$ are the fundamental matrix solutions to equations \eqref{eq:hillsys} and \eqref{eq:rot}, respectively.

Next, for each $\lambda\in\R$ and $s \in (0,L]$, let us define the trace map $\Phi^{\lambda}_s: Y_{s,\lambda} \to \R^{16n}$ by the following formula: 
\begin{equation}\label{eq:trace}
 \Phi^{\lambda}_s:  \big(\bp, \bw\big)^\top \mapsto \big( \bp(-s), \bw(-s),  \bp(s), \bw(s) \big)^\top \in \R^{16n}.
\end{equation}
We remark that $ \Phi^{\lambda}_s$ can be identified with the following $(16n\times 8n)$ matrix,
\begin{equation}\label{Defpsihat}
\widehat{\Phi}_s^\lambda=\begin{pmatrix}\Psi_A(-s,\lambda)&0_{4n}\\
0_{4n}&\Psi_B(-s,s)\\\Psi_A(s,\lambda)&0_{4n}\\
0_{4n}&\Psi_B(s,s)\end{pmatrix},
\end{equation} since for the solution $\big(\bp, \bw\big)^\top\in Y_{s,\lambda}$ given by $\bp(x)=\Psi_A(x,\lambda)\bp(-L)$ and $\bw(x)=\Psi_B(x,s)\bw(-L)$, clearly,  the vector $\Phi^{\lambda}_s\big((\bp, \bw)^\top\big)\in\R^{16n}$ is the product of the matrix $\widehat{\Phi}_s^\lambda$ and the vector $\big(\bp(-L),\bw(-L)\big)^\top\in\R^{8n}$.

Let us introduce the $(16n\times 16n)$ orthogonal skew-symmetric matrix $\Omega$ (and thus a symplectic structure on $\R^{16n}$) by the formula
\begin{equation}\label{eq:omega}
 \Omega = (J \otimes I_{2n}) \oplus (J^\top\otimes I_{2n}) \oplus (J^\top \otimes I_{2n}) \oplus (J\otimes I_{2n}),\quad J=\begin{pmatrix} 0 & 1 \\ -1 & 0 \end{pmatrix},
 \end{equation}
where $J$  is the standard symplectic matrix.
\begin{theorem}\label{th:lagrange}
For all  $s \in (0,L]$ and $\lambda \in \R$ the plane $\Phi^{\lambda}_s(Y_{s,\lambda})$ belongs to the space $\Lambda(8n)$ of Lagrangian $8n$-planes in $\R^{16n}$, with the Lagrangian structure $\omega(v_1,v_2)=\langle v_1,\Omega v_2\rangle_{\R^{16n}}$ given by $\Omega$ defined in \eqref{eq:omega}.
\end{theorem}
\begin{proof}
Equations \eqref{eq:hillp} and \eqref{eq:rotshort} are Hamiltonian, with the symplectic structure defined by the matrices 
\begin{equation}\label{eq:newJ}
J_n:=J\otimes I_{2n}\,\text{ and }\, J^{(n)}:=J^\top\otimes I_{2n}, 
\end{equation}
respectively. In particular,
\begin{equation}\label{symm}
J_nA(x,\lambda)=\big(J_nA(x,\lambda)\big)^\top,\quad
J^{(n)}B(s,\theta)=\big(J^{(n)}B(s,\theta)\big)^\top.
\end{equation}
Writing \eqref{eq:omega} as 
\begin{equation}\label{newOm}
\Omega=J_n\oplus J^{(n)}\oplus (-J_n)\oplus(-J^{(n)}),
\end{equation} for any two vectors from $\Phi_s^\lambda(Y_{s,\lambda})$, $v_1=\big(\bp_1(-s), \bw_1(-s), \bp_1(s), \bw_1(s)\big)^\top$ and $v_2=\big(\bp_2(-s), \bw_2(-s), \bp_2(s), \bw_2(s)\big)^\top$, we infer: 
\begin{align*}
\langle v_1,\Omega& v_2\rangle_{\R^{16n}}=
\langle \bp_1(-s),J_n \bp_2(-s)\rangle_{\R^{4n}}+
\langle \bw_1(-s),J^{(n)} \bw_2(-s)\rangle_{\R^{4n}}\\
&\qquad-\langle \bp_1(s),J_n \bp_2(s)\rangle_{\R^{4n}}-
\langle \bw_1(s),J^{(n)} \bw_2(s)\rangle_{\R^{4n}}\\
&=\int^{-s}_s\frac{d}{dx}\Big(\langle \bp_1(x),J_n \bp_2(x)\rangle_{\R^{4n}}+
\langle \bw_1(x),J^{(n)} \bw_2(x)\rangle_{\R^{4n}}\Big)\,dx\\
&=\int^{-s}_s\Big(\langle \bp_1'(x),J_n \bp_2(x)\rangle_{\R^{4n}}+\langle \bp_1(x),J_n \bp_2'(x)\rangle_{\R^{4n}}\\&\qquad
+\langle \bw_1'(x),J^{(n)} \bw_2(x)\rangle_{\R^{4n}}
+\langle \bw_1(x),J^{(n)} \bw_2'(x)\rangle_{\R^{4n}}\Big)\,dx\\
&=\int^{-s}_s\Big(\langle A(x,\lambda)\bp_1(x),J_n \bp_2(x)\rangle_{\R^{4n}}+\langle \bp_1(x),J_n A(x,\lambda)\bp_2(x)\rangle_{\R^{4n}}\\&\qquad
+\langle B(s,\theta)\bw_1(x),J^{(n)} \bw_2(x)\rangle_{\R^{4n}}
+\langle \bw_1(x),J^{(n)} B(s,\theta)\bw_2(x)\rangle_{\R^{4n}}\Big)\,dx\\
&=\int^{-s}_s\Big(-\langle J_nA(x,\lambda)\bp_1(x), \bp_2(x)\rangle_{\R^{4n}}+\langle \bp_1(x),J_n A(x,\lambda)\bp_2(x)\rangle_{\R^{4n}}\\&\qquad
-\langle J^{(n)} B(s,\theta)\bw_1(x), \bw_2(x)\rangle_{\R^{4n}}
+\langle \bw_1(x),J^{(n)} B(s,\theta)\bw_2(x)\rangle_{\R^{4n}}\Big)\,dx\\
&=0,\end{align*} 
where in the last two lines we used  $(J_n)^\top=-J_n$,
$(J^{(n)})^\top=-J^{(n)}$ and \eqref{symm}.
\end{proof}
 We remark that $X \times X$ with $X$ defined in \eqref{DefX1} is a Lagrangian plane in $\R^{16n}$ with the same symplectic structure given by $\Omega$ (indeed, this was why $\Omega$ was chosen in the first place). This can be verified by a straightforward calculation. 

\begin{definition}\label{def3.5}
  For a given $\lambda$, a point $s\in(0,L]$ is called a {\em ($\lambda$-)conjugate point} if $\Phi^{\lambda}_s(Y_{s,\lambda})\in \train(X \times X)$,
  where $X$ is defined in \eqref{DefX1}. 
  \end{definition}
The latter inclusion means that there exists a solution of the system of equations \eqref{eq:hillsys}, \eqref{eq:rot} on the segment $[-s,s]$ satisfying
the boundary conditions \eqref{eq:pereigen} with $L$ replaced by $s$, that is, the boundary condtions
\begin{equation}\label{eq3.28}
\begin{pmatrix} p(s) \\ q(s) \end{pmatrix} = \begin{pmatrix} I_n \otimes U(\theta)  & 0 \\ 0 & I_n \otimes U(\theta) \end{pmatrix} \begin{pmatrix} p(-s) \\ q(-s) \end{pmatrix}, 
\end{equation} 
and the boundary conditions \eqref{eq:wpereigen}.

Our next objective is to relate the crossings of the path $\big\{\Phi_\lambda^s(Y_{s,\lambda})\big\}$ and eigenvalues of differential operators $H_{\theta,s}$ in $L^2([-s,s])$ introduced as follows, cf.\ \eqref{defLtheta}. 
For any $s\in(0,L]$ and $\theta\in[0,2\pi]$, let $H_{\theta,s}=\frac{d^2}{dx^2}+V(x)$ with 
\begin{equation}\lb{defLthetaS}\begin{split}
 \dom(H_\theta)=\Big\{&
y\in L^2([-s,s])\Big|
 y,y'\in AC_{\rm loc}([-s,s]),\\ & y''\in L^2([-s,s])\text{ and the boundary condition \eqref{sthetaBC} holds}\Big\}.
\end{split}\end{equation}
In particular, $H_{\theta,L}=H_\theta$.
We remark that $y\in\ker\big(H_{\theta,s}-\lambda I_{L^2([-s,s])}\big)$ if and only if the vector valued function $\bp=(p,q)^\top$ defined in \eqref{Defpq} is a solution of \eqref{eq:hillp} on $[-s,s]$ that satisfies the boundary conditions \eqref{eq3.28}.
\begin{definition}
We say that $\lambda$ is an {\em $(\theta,s)$-eigenvalue} of equation (\ref{eq:hill}) if there is a nonzero solution of (\ref{eq:hillsys}) such that the boundary conditions (\ref{eq3.28}) are satisfied. 
\end{definition}
Recall that $\Psi_A(x,\lambda)$ is the fundamental matrix solution of the system \eqref{eq:hillsys}, and $M_A(s,\lambda)=\Psi_A(s,\lambda)\Psi_A(-s,\lambda)^{-1}$ is the propagator for $s\in(0,L]$ so that $\bp(s)=M(s,\lambda)\bp(-s)$ for a solution of \eqref{eq:hillp}.  Also, we recall that the multiplicity 
 of the eigenvalue $\lambda$ of the operator $H_\theta$ is the (complex) dimension of the solution space of the boundary value problem
 \eqref{eq:hill}, \eqref{thetaBC} on $[-L,L]$.
 
\begin{proposition}\label{cor:mult}
For any $\lambda\in\R$, $\theta\in[0,2\pi]$, and $s\in(0,L]$ the following assertions are equivalent:

 $(i)$\, $\lambda\in\Sp(H_{\theta,s})$ in $L^2([-s,s])$;

 $(ii)$\, $e^{i\theta}\in\Sp\big(M_A(s,\lambda)\big)$;

$(iii)$\, $s$ is a $\lambda$-conjugate point, that is, $\Phi_s^\lambda(Y_{s,\lambda})\in\train(X\times X)$.
 
\noindent  Moreover, the multiplicity 
 of the eigenvalue $\lambda$ of the operator $H_{\theta,s}$ is equal to the real dimension of the subspace $\Phi^{\lambda}_s(Y_{s,\lambda})\cap(X \times X)$. In particular, $\lambda$ is a $\theta$-eigenvalue of \eqref{eq:hill} if and only if $L$ is a $\lambda$-conjugate point, that is, $\Phi^{\lambda}_L(Y_{L,\lambda}) \in \train(X \times X)$, and $\lambda$ is an $(\theta,s)$-eigenvalue of \eqref{eq:hill} if and only $\Phi^{\lambda}_s(Y_{s,\lambda}) \in \train(X \times X)$.
 \end{proposition}
 
 \begin{proof} $(i) \Rightarrow (ii)$\, Take a nonzero $y\in\ker\big(H_{\theta,s}-\lambda I_{L^2([-s,s])}\big)$ and let $\by=(y,y')^\top$
 be the complex valued $(n\times 1)$ solution of the first order system
 \begin{equation}\label{CC}
 \by'=A^\bbC (x,\lambda)\by,\quad A^\bbC (x,\lambda)=\begin{pmatrix} 0_n&I_n\\ \lambda I_n-V(x)&0_n\end{pmatrix}
 \end{equation} 
 that satisfies the boundary condition \eqref{sthetaBC}. Let $\Psi_A^\bbC(x,\lambda)$ denote the fundamental matrix solution to \eqref{CC} such that $\Psi_A^\bbC(-L,\lambda)=I_n$, so that $\by(x)=\Psi_A^\bbC(x,\lambda)\by(-L)$, and let $M_A^\bbC(s,\lambda)=\Psi_A^\bbC(s,\lambda)\Psi_A^\bbC(-s,\lambda)^{-1}$ denote the propagator such that $\by(s)=M_A^\bbC(s,\lambda)\by(-s)$. Due to \eqref{sthetaBC}, we have $e^{i\theta}\in\Sp\big(M_A^\bbC(s,\lambda)\big)$. Let $T:\bbC^{2n}\to\bbR^{4n}$ be the map $\by\mapsto\bp=(p,q)^\top$ defined in \eqref{Defpq}. Then $\Psi_A(x,\lambda)=T\Psi_A^\bbC(x,\lambda)T^{-1}$ and $M_A(s,\lambda)=TM_A^\bbC(s,\lambda)T^{-1}$, yielding $(ii)$.
 
 $(ii) \Rightarrow (iii)$\, For a vector $\bv\in\bbC^{2n}$ satisfying $M_A^\bbC(s,\lambda)\bv=e^{i\theta}\bv$ let $\by(x)=\Psi_A^\bbC(x,\lambda)\bv$ be the solution of \eqref{CC} satisfying \eqref{sthetaBC}.
 Using \eqref{Defpq}, construct the solution $\bp$ of \eqref{eq:hillp} satisfying \eqref{eq3.28}, that is, satisfying $\bp(s)=(I_{2n}\otimes U(\theta))\bp(-s)$. Pick the solution $\bw$ of \eqref{eq:rotshort} such that $\bw(-s)=\bp(-s)$. Since solutions of \eqref{eq:rotshort} automatically satisfy \eqref{eq:wpereigen}, we have $\bp(\pm s)=\bw(\pm s)$ and thus $\Phi_s^\lambda(Y_{s,\lambda})\in\train(X\times X)$.
 
 $(iii) \Rightarrow (i)$\, Pick a solution $(\bp,\bw)^\top$ of \eqref{eq:augodeshort} such that $\Phi_s^\lambda\big((\bp,\bw)^\top\big)\in X\times X$; then $\bp(\pm s)=\bw(\pm s)$. Since $\bw$ automatically satisfies \eqref{eq:wpereigen}, the boundary condition $\bp(s)=(I_{2n}\otimes U(\theta))\bp(-s)$ holds. It follows that the solution  $\by$ of \eqref{CC} related to $\bp=(p,q)^\top$ via \eqref{Defpq} satisfies
 the boundary condition \eqref{sthetaBC}, thus yielding $H_{\theta,s}y=\lambda y$.
 
 To prove the equality of the multiplicity and the dimension of the intersection, we remark that the map $y\mapsto (\bp(-s),\bw(-s),\bp(s),\bw(s))^\top\in\bbR^{16n}$ from the finite dimensional (complex) space
 \begin{equation}\lb{defLthetaker}\begin{split}
 \ker(H_{\theta,s}-\lambda I_{L^2([-s,s])})=\Big\{
y&\in L^2([-s,s])\Big|
 y,y'\in AC_{\rm loc}([-s,s]),\\ & \text{ and \eqref{eq:hill},\,\eqref{sthetaBC} hold}\Big\}
\end{split}\end{equation}
into the finite dimensional space $\Phi_s^\lambda(Y_{s,\lambda})\cap(X\times X)$ has zero kernel and is an isomorphism. Thus $2\dim_\bbC\ker(H_{\theta,s}-\lambda I_{L^2([-s,s])})=\dim_\bbR\big(\Phi_s^\lambda(Y_{s,\lambda})\cap(X\times X)\big)$.
\end{proof}

 Since the boundary value problem on the segment $[-s,s]$ makes sense only for positive $s$, we may restrict $s$ to $s\in[s_0,L]$ for some $s_0>0$.  Since the operator $H_{\theta,s}$ is semibounded from above, for a $\lambda_\infty$ large enough there are no $(\theta,s)$-eigenvalues with  $\lambda\ge\lambda_\infty$. Therefore, we may restrict $\lambda$ to $\lambda\in[0,\lambda_\infty]$.
  As we will see in Lemma \ref{lem:llss}, for $\lambda_\infty$ large enough there are no $s\in[s_0,L]$ such that $\Phi^{\lambda_\infty}_{s}(Y_{s,\lambda}) \in \train(X \times X)$ provided $\theta\in[0,2\pi]$, and for $s_0$ small enough there are no $\lambda\in[0,\lambda_\infty]$ such that $\Phi^{\lambda}_{s_0}(Y_{s_0,\lambda}) \in \train(X \times X)$ provided $\theta\in(0,2\pi)$. 
   
With no loss of generality (by varying $\theta$ a little, if needed),  we may assume that $\lambda=0$ is {\em not} a $\theta$-eigenvalue for a given $\theta$, see Figures \ref{F1} and  \ref{F2}. (This ensures that all crossings take place away from the upper left corner in Figure \ref{F2}). This is not actually necessary, but more of a convenience. We can simply use the crossing form calculation at the endpoints if there is a crossing at the upper left corner, taking into account half of the local Maslov index each time. 
 
We also remark that for a fixed $\lambda_\infty $, and $s_0$, we can view the map $\Phi^{\lambda}_s$ as a continuous map from the square $\Phi^{\lambda}_s: [0,\lambda_\infty] \times [s_0,L] \to \Lambda(8n)$ to the space of Lagrangian planes, see Figure \ref{F2}. 
 As such, its image must be homotopic to a point, and so we have the following theorem.

\begin{theorem} \label{th:mas1}The homotopy class of the image of the boundary of the square $[0,\lambda_\infty] \times [s_0,L] $ under the map $\Phi$ is zero in $\pi_1(\Lambda(8n))$. 
\end{theorem}

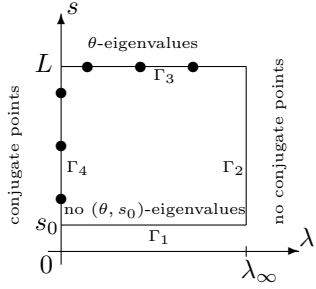
\begin{figure}
\begin{picture}(100,100)(-20,0)
\put(2,2){0}
\put(10,5){\vector(0,1){95}}
\put(5,10){\vector(1,0){95}}
\put(71,40){\text{\tiny $\Gamma_2$}}
\put(12,40){\text{\tiny $\Gamma_4$}}
\put(45,75){\text{\tiny $\Gamma_3$}}
\put(43,14){\text{\tiny $\Gamma_1$}}
\put(100,12){$\lambda$}
\put(12,100){$s$}
\put(80,20){\line(0,1){60}}
\put(80,8){\line(0,1){4}}
\put(78,0){$\lambda_\infty$}
\put(10,20){\line(1,0){70}}
\put(10,80){\line(1,0){70}}
\put(0,78){$L$}
\put(0,18){$s_0$}
\put(10,30){\circle*{4}}
\put(10,50){\circle*{4}}
\put(10,70){\circle*{4}}
\put(20,80){\circle*{4}}
\put(40,80){\circle*{4}}
\put(60,80){\circle*{4}}
\put(20,87){{\tiny \text{$\theta$-eigenvalues}}}
\put(11,24){{\tiny \text{no $(\theta,s_0)$-eigenvalues}}}
\put(-10,20){\rotatebox{90}{{\tiny conjugate points}}}
\put(90,20){\rotatebox{90}{{\tiny no conjugate points}}}
\end{picture}
\caption{ $\lambda=0$ is {\em not} a $\theta$-eigenvalue, $\theta\in(0,2\pi)$, and $s_0$ is small enough}\label{F2}
\end{figure}

It is well known that $\pi_1(\Lambda(8n)) \approx \Z$, and that the class of a closed curve can be determined by the number of intersections of such a curve (up to homotopy) with the train of a fixed Lagrangian plane (see for example, \cite{arnold67}, or \cite{rs93} and the references therein). Denote by $\Gamma$ the boundary of the image of $[0,\lambda_\infty] \times [s_0,L]$ under $\Phi^{\lambda}_s$. The key idea here is that under the construction given above, we have an eigenvalue interpretation for the intersection of $\Phi^\lambda_s(Y_{s,\lambda})$ with the train of a special plane. Since the signed number of intersections does not change under homotopy, and $\Gamma$ is homotopic to a point, we have the following result. 

\begin{corollary}\label{cor3.9} As we travel along $\Gamma$, the signed number of intersections of $\Gamma$ with $\train(X\times X)$, counted with multiplicity, is equal to zero. 
\end{corollary}

\begin{remark}\label{remN3.9}
It is convenient for us to break up the curve $\Gamma=\Gamma_{(\theta,s_0)}$ into the four pieces corresponding to the sides of the square from which it comes. Let $\Gamma_1=\Gamma_{1,(\theta,s_0)}$ denote $\big\{\Phi_{s_0}^{\lambda}(Y_{s_0,\lambda})\big|\, \lambda\in[0, \lambda_\infty]\big\}$, let $\Gamma_2=\Gamma_{2,(\theta,s_0)}$ denote $\big\{\Phi^{\lambda_\infty}_{s}(Y_{s,\lambda_\infty})\big|\, s\in[s_0,L]\big\}$, let $\Gamma_3=\Gamma_{3,(\theta,s_0)}$ denote $\big\{\Phi_L^{\lambda}(Y_{L,\lambda})\big|\, \lambda\in[\lambda_\infty,0]\big\}$, and let $\Gamma_4=\Gamma_{4, (\theta,s_0)}$ denote $\big\{\Phi^{0}_{s}(Y_{s,0})\big|\, s\in[L, s_0]\big\}$. 
\end{remark}

Let $A_i=A_{i,(\theta,s_0)}$ denote the Maslov index of each piece of $\Gamma_i$, as defined in \eqref{DefMI}, that is, 
\begin{equation}\label{defAi}
 A_i := \mi(\Gamma_i,X\times X).
 \end{equation}
We will also denote by $B_i=B_{i,(\theta,s_0)}$ the following expression: 
\begin{equation}\label{defBi}\begin{split}
 B_i &:=   \left| \frac{1}{2} \sgn Q_\cM(\Gamma_i(a_i),X\times X) \right|  
 \\ & + \sum_{t \in (a_i,b_i)} \left| \sgn Q_\cM(\Gamma_i(t),X\times X) \right|  + \left| \frac{1}{2} \sgn Q_\cM(\Gamma_i(b_i),X\times X) \right|,
 \end{split}\end{equation}
 where $\Gamma_i(a_i)$ and $\Gamma_i(b_i)$ denote the endpoints of the curve $\Gamma_i$.
That is, $B_i$ is the number of crossings along $\Gamma_i$ each counted regardless of sign, but taking into account the multiplicity of crossings. 
For instance, if we have three simple crossings on $\Gamma_i$ with signs $+,-,+$, we would have that $A_i = 1$, while $B_i = 3$. It is also worth noting that in all cases $|A_i |\leq B_i$.

We will now show that $B_2=0$ 
provided $\lambda_\infty$ is large enough and $\theta\in[0,2\pi]$, and that $B_1=0$ provided
$s_0>0$ is small enough and $\theta\in(0,2\pi)$. For $\theta=0$, see a computation of $B_1$ in Lemma \ref{llssnew}. 
We will repeatedly use the following elementary fact.
\begin{theorem}\label{KatoTh}{\bf\cite[Theorem V.4.10]{Kato}}
Let $\cH$ be selfadjoint and $\cV\in\cB(\cX)$ be symmetric operators   on a Hilbert space $\cX$. Then $$\dist\big(\Sp(\cH+\cV),\Sp(\cH)\big)\le\|\cV\|_{\cB(\cX)}.$$
\end{theorem} 
\noindent We recall that $V$ is a bounded matrix valued  function on $[-L,L]$
and denote the supremum of its matrix norm by 
$\|V\|_\infty=\sup_{x\in[-L,L]}\|V(x)\|_{\bbR^n\times\bbR^n}$.
\begin{lemma}\label{lem:llss} \mbox{}
\begin{itemize}
\item[(i)] \label{llsshyp1} Assume that $\theta\in[0,2\pi]$. If $\lambda_\infty>\|V\|_\infty$ then $B_2=0$. 
\item[(ii)] \label{llsshyp2} Assume that $\theta\in(0,2\pi)$. If $\lambda_\infty>\|V\|_\infty$ and  $$s_0<\frac12{\min\big\{\theta, 2\pi-\theta\big\}}\,{(\|V\|_\infty+\lambda_\infty)^{-1/2}},$$ then $B_1=0$.
\end{itemize}
\end{lemma}
\begin{proof}
Let $H_{\theta,s}^{(0)}=\frac{d^2}{dx^2}$ with $\dom(H_{\theta,s}^{(0)})=\dom(H_{\theta,s})$.
The eigenvalues of $H_{\theta,s}^{(0)}$ are given by the formula
\begin{equation} \label{formmu}
\mu_k=-\Big(\frac{\theta+2\pi k}{2s}\Big)^2, \, \, k\in\bbZ.
\end{equation}
Indeed, inserting  the general solution $y(x)=c_1e^{\sqrt{\mu}x}+c_2e^{-\sqrt{\mu}x}$ of the equation $y''=\mu y$ in the boundary conditions \eqref{sthetaBC}, we obtain the system of equations for $c_1$, $c_2$, 
whose determinant must be equal to zero, yielding \eqref{formmu}. 

For $s\in(0,L]$ and $\mu_k$ in \eqref{formmu} we denote $\mu(s)=\max_{k\in\bbZ}\mu_k$. Then \eqref{formmu} implies 
\begin{equation}
\mu(s)=-\big(\min\big\{\theta,2\pi-\theta\big\}/(2s)\big)^2, 
\quad \theta\in[0,2\pi],\,s\in(0,L],
\end{equation}
and 
\begin{equation}\label{bbsp}
\Sp(H_{\theta,s}^{(0)})\subset\big(-\infty,\mu(s)\big]\subset(-\infty,0], \,\text{ for each } s\in(0,L].
\end{equation}
By Theorem \ref{KatoTh} we infer:
\begin{equation}\label{KatTh}
\dist\big(\Sp(H_{\theta,s}),\, \Sp(H_{\theta,s}^{(0)})\big)\le\|V\|_{\cB(L^2([-s,s]))}\le\|V\|_\infty.
\end{equation}
 This and the second inclusion in \eqref{bbsp}
yield  $\Sp(H_{\theta,s})\subset(-\infty, \|V\|_\infty]$. If $s$ is a conjugation 
point for a given $\lambda$, then there is a solution $y$ of the equation 
$H_{\theta,s}y=\lambda y$ satisfying \eqref{sthetaBC}, that is, $\lambda$ is an eigenvalue of $H_{\theta,s}$. Thus, there are no conjugation points for $\lambda_\infty$ provided $\lambda_\infty>\|V\|_\infty$, proving assertion $(i)$. 

$(ii)$\,\, Assume that $\theta\in(0,2\pi)$, fix $\lambda_\infty>\|V\|_\infty$, and consider any $\lambda\in[0,\lambda_\infty]$ and $s\in(0,L]$. If $y$ is a solution of the equation $y''+V(x)y=\lambda y$ for $|x|\le s$ satisfying  boundary conditions \eqref{sthetaBC}  then $z(x)=y(sx/L)$ for $|x|\le L$ satisfies the equation 
\begin{equation}
H_{\theta,L}^{(1)}z:=z''+\Big(\big(s/L\big)^2V(sx/L)-\lambda(s/L)^2\Big)z=0,\, x\in[-L,L],
\end{equation}
and boundary conditions \eqref{thetaBC}. In other words, $\lambda$ is an eigenvalue of $H_{\theta,s}$ on $L^2([-s,s])$ if and only if zero is an eigenvalue of $H_{\theta,L}^{(1)}$ on $L^2([-L,L])$.
Since the potential in $H_{\theta,L}^{(1)}$ is $\big(s/L\big)^2V(s(\cdot)/L)-\lambda(s/L)^2$, by Theorem \ref{KatoTh} we infer:
\begin{align}
\dist\big(\Sp(H_{\theta,L}^{(1)}),\, \Sp(H_{\theta,L}^{(0)})\big)&\le\|\big(s/L\big)^2V(s(\cdot)/L)-\lambda(s/L)^2\|_{\cB(L^2([-L,L]))}\nonumber\\
&\le\big(s/L\big)^2\big(\|V\|_\infty+\lambda_\infty\big).\label{KatTh1}
\end{align}
This and the first inclusion in \eqref{bbsp} with $s=L$ imply 
\begin{equation}
\Sp(H_{\theta,L}^{(1)})\subset\Big(-\infty,\,
\big({s}/{L}\big)^2\big(\|V\|_\infty+\lambda_\infty\big)+\mu(L)\Big],
\end{equation}
where $\mu(L)<0$ due to $\theta\in(0,2\pi)$.
In particular, using \eqref{formmu}, if $$s_0^2\big(\|V\|_\infty+\lambda_\infty\big)<(\min\big\{\theta, 2\pi-\theta\big\}/2)^2$$
then zero is not an eigenvalue of $H_{\theta,L}^{(1)}$ and thus $\lambda$ is not an eigenvalue of $H_{\theta,s_0}$ on $L^2([-s_0,s_0])$, as needed in $(ii)$.
\end{proof}

Alternatively, one can prove that for $\theta\in(0,2\pi)$ there are no conjugate points, provided $s$ is sufficiently small, using Proposition \ref{cor:mult} $(ii)$: Since $\Sp\big(M_A(s,\lambda)\big)\to\{1\}$ as $s\to0$, we infer that $e^{i\theta}\notin \Sp\big(M_A(s,\lambda)\big)$ for 
$s$ small enough.

The periodic case $\theta=0$ or $\theta=2\pi$ is somehow special and should be treated separately. Since the periodic boundary conditions \eqref{thetaBC}, \eqref{sthetaBC} hold when either $\theta=0$ or $\theta=2\pi$, we will conclude this section by considering the case $\theta=0$ (see also Lemma \ref{llssnew} for more information regarding this case). 

We will begin by constructing the curve $\Gamma=\Gamma_{(0,0)}$ for $\theta=0$ and $s_0=0$ (note that the construction described in Remark \ref{remN3.9} does not work as \eqref{eq:rotshort} is not defined for $s=0$). If $\theta=0$ then $\mathfrak{u}(s,0)=0_2$ in \eqref{eq:rot}  and $B(s,0)=0_{4n}$ in \eqref{eq:rotshort} for all $s>0$.
Thus, we have $\Psi_B(x,s)=I_{4n}$ for $\theta=0$ and $s>0$. Letting
$\Psi_B(x,0)=I_{4n}$ for $s=0$ and all $x\in[-L,L]$, we can extend $\Psi_B(x,s)$ continuously from $s>0$ to $s=0$ although the differential equation \eqref{eq:rotshort} is not defined for $s=0$. This allows us to define the curve $\Gamma_1$ for $\theta=0$ and $s_0=0$ as follows: Recall that
the curve $\Gamma_1=\Gamma_{1,(\theta,s_0)}$ is defined via \eqref{Defpsihat} as the set
\begin{equation}\label{defG1c}
\Gamma_{1,(\theta,s_0)}=\big\{ \widehat{\Psi}_{s_0}^\lambda\bv\big|\,
\lambda\in[0,\lambda_\infty], \bv\in\bbR^{8n}\big\}.
\end{equation}
Setting $\theta=0$ and passing in \eqref{Defpsihat} to the limit yields $\widehat{\Psi}_{s_0}^\lambda\to\widehat{\Psi}_{0}^\lambda$ as $s_0\to0^+$ uniformly for $\lambda\in[0,\lambda_\infty]$, where we define
\begin{equation}\label{Defpsihat0}
\widehat{\Phi}_0^\lambda=\begin{pmatrix}\Psi_A(0,\lambda)&0_{4n}\\
0_{4n}&I_{4n}\\\Psi_A(0,\lambda)&0_{4n}\\
0_{4n}&I_{4n}\end{pmatrix}.
\end{equation}
Letting
\begin{equation}\label{defG1c0}
\Gamma_{1,(0,0)}=\big\{ \widehat{\Psi}_{0}^\lambda\bv\big|\,
\lambda\in[0,\lambda_\infty], \bv\in\bbR^{8n}\big\},
\end{equation}
we thus introduce the curve $\Gamma_1=\Gamma_{1,(0,0)}$ for $\theta=0$ and $s_0=0$.
This curve is homotopic to the curve $\Gamma_1=\Gamma_{1,(\theta,s_0)}$ for $\theta>0$ and $s_0>0$ although the endpoints of the two curves are not fixed. A direct computation shows that 
$\big(\widehat{\Psi}_{0}^\lambda\big)^\top \Omega \widehat{\Psi}_{0}^\lambda=0_{16n}$, and thus $\Gamma_{1,(0,0)}$ is a curve in $\Lambda(8n)$. Clearly, $\Gamma_{1,(0,0)}$ lies in $\train(X\times X)$, and thus is not regular. This makes the computation of $\mi(\Gamma_1, X\times X)$ for $\theta=0$ and $s_0=0$ with this choice of $\Gamma_1$ difficult.
By appending to $\Gamma_{1,(0,0)}$ the three remaining curves $\Gamma_{j,(0,0)}$, $j=2,3,4$, corresponding to the remaining three sides of the square $[0,L]\times[0,\lambda_\infty]$,
we construct the entire curve $\Gamma=\Gamma_{(0,0)}$ for $\theta=0$ and $s_0=0$ which is homotopic to the curve $\Gamma=\Gamma_{(\theta,s_0)}$ for $\theta>0$ and $s_0>0$. We can appeal to a theorem in \cite{rs93} which says that every continuous curve is homotopic to a curve with only regular crossings. Thus we can compute the Maslov index of $\Gamma=\Gamma_{(0,0)}$ and verify that it is indeed 0, that is, that Corollary \ref{cor3.9} holds for $\theta=0$ and $s_0=0$. 

One can also define conjugate point as a point $s$ where $e^{i\theta}\in\Sp\big(M_A(s,\lambda)\big)$, see Proposition \ref{cor:mult} $(ii)$. Unlike  Definition \ref{def3.5}, this latter definition is applicable for $s=0$ as well. But for $\theta=0$, since  $M_A(0,\lambda)=I_{4n}$, we have that $1=e^{i0}\in\Sp\big(M_A(0,\lambda)\big)$, and thus $s=0$ is the conjugate point for all $\lambda\in[0,\lambda_\infty]$. In particular, for $\theta=0$ the curve $\Gamma_2$ has a conjugate point at $s=0$.

We summarize the discussion as follows and refer to Lemma \ref{llssnew} for more information regarding the case $\theta=0$.
\begin{corollary}\label{cor:3.10} Assume that $\theta=0$ and that $\Gamma$ is the curve just defined for $s_0=0$ using \eqref{Defpsihat0}, and parametrized by the sides of the square $[0,L]\times[0,\lambda_\infty]$. Then $\mi(\Gamma, X\times X)=0$. Each point of the curve $\Gamma_1$ belongs to $\train(X\times X)$. The lower endpoints of the curves $\Gamma_2$ and $\Gamma_4$, and all points of $\Gamma_1$ are conjugate points in the sense that $1=e^{i0}\in\Sp\big(M_A(0,\lambda)\big)$ for  all $\lambda\in[0,\lambda_\infty]$. 
\end{corollary}

\section{Monotonicity of the Maslov index}\label{mmi}

We will now establish monotonicity of the Maslov index with respect to the parameter $\lambda$ and, under some additional assumptions, with respect to the parameter $s$. 
Let us begin with $\lambda$. We recall from Remark \ref{remN3.9} that the curve $\Gamma_3$ is parametrized by the parameter $\lambda$ {\em decaying} from $\lambda_\infty$ to $0$ while the curve $\Gamma_1$ is parametrized by the parameter $\lambda$ {\em growing} from $0$ to $\lambda_\infty$. The strategy of the proof of the next result follows the proof of \cite[Lemma 4.7]{DJ11}.
\begin{lemma}\label{lem:cross}
For any $\theta\in[0,2\pi]$ and any fixed $s\in(0,L]$, each crossing $\lambda_0\in(0,\lambda_\infty)$ of the path $\big\{\Phi_s^\lambda(Y_{s,\lambda})\big\}_{\lambda=\lambda_0-\varepsilon}^{\lambda_0+\varepsilon}$, with $\varepsilon>0$ small enough, is negative.  In particular,  if $0\notin\Sp(H_\theta)$, then $B_3=A_3$
 and if $0\notin\Sp(H_{\theta,s_0})$ then $B_1=-A_1$.
 \end{lemma}
\begin{proof}
Let $\lambda_0\in(0,\lambda_\infty)$ be a crossing, so that $\Phi_s^{\lambda_0}(Y_{s,\lambda_0})\cap(X\times X)\neq\{0\}$. Let $V^\bot$ be a subspace in $\R^{16n}$ transversal to $\Phi_s^{\lambda_0}(Y_{s,\lambda_0})$.
Then $V^\bot$ is transversal to $\Phi_s^{\lambda}(Y_{s,\lambda})$ for all $\lambda\in[\lambda_0-\varepsilon,\lambda_0+\varepsilon]$ for $\varepsilon>0$ small enough. Thus, there exists a smooth family of matrices, $\phi(\lambda)$, for $\lambda\in[\lambda_0-\varepsilon,\lambda_0+\varepsilon]$, viewed as operators $\phi(\lambda): \Phi_s^{\lambda_0}(Y_{s,\lambda_0})\to V^\bot$, such that $\Phi_s^{\lambda}(Y_{s,\lambda})$ is the graph of $\phi(\lambda)$. Fix any nonzero ${\rm v}\in\Phi_s^{\lambda_0}(Y_{s,\lambda_0})\cap(X\times X)$ and consider the curve $v(\lambda)={\rm v}+\phi(\lambda){\rm v}\in\Phi_s^{\lambda}(Y_{s,\lambda})$ for $\lambda\in[\lambda_0-\varepsilon,\lambda_0+\varepsilon]$ with $v(\lambda_0)={\rm v}$. By the definition of $Y_{s,\lambda}$, there is a family of solutions
$(\bp(\cdot\,,\lambda),\bw(\cdot\,,\lambda))^\top$ of \eqref{eq:augodeshort} such that
$v(\lambda)=\Phi_s^{\lambda}\big((\bp(\cdot\,,\lambda),\bw(\cdot\,,\lambda))^\top\big)$. We claim that 
\begin{equation}\label{negdef}
\omega\big(v(\lambda_0),\,\frac{\partial v}{\partial\lambda}(\lambda_0)\big)<0.
\end{equation}
Assuming the claim, we finish the proof as follows: Since for each  nonzero ${\rm v}\in\Phi_s^{\lambda_0}(Y_{s,\lambda_0})\cap(X\times X)$ the crossing form $Q_\cM$ satisfies
\begin{align*}
Q_\cM({\rm v},{\rm v})&=\frac{d}{d\lambda}\Big|_{\lambda=\lambda_0}\omega({\rm v},\phi(\lambda){\rm v})=
\frac{d}{d\lambda}\Big|_{\lambda=\lambda_0}\omega({\rm v},{\rm v}+\phi(\lambda){\rm v})\\&=\omega\big(v(\lambda_0),\,\frac{\partial v}{\partial\lambda}(\lambda_0)\big)<0,
\end{align*}
the form is negative definite. Thus, the crossing $\lambda_0\in(0,\lambda_\infty)$ is negative. In particular, taking into account that the path $\Gamma_3=\big\{\Phi^\lambda_L(Y_{L,\lambda})\big\}_{\lambda=\lambda_\infty}^0$ is parametrized by the parameter $\lambda$ {\em decaying} from $\lambda_\infty$ to $0$, each crossing $\lambda_0$ along $\Gamma_3$ is positive. Thus, the Maslov index $A_3$ of the path $\Gamma_3$ is equal to $B_3$.
Taking into account the parametrization of $\Gamma_1=\big\{\Phi_{s_0}^\lambda\big\}_{\lambda=0}^{\lambda_\infty}$,
a similar argument yields $A_1=-B_1$.

Starting the proof of claim \eqref{negdef}, for the solution $\bp=\bp(x,\lambda)$ we compute the $\lambda$-derivative (for brevity, denoted below by dot) in equation \eqref{eq:hillp}, and obtain the equation
\begin{equation}\label{lder}
\dot{\bp}'(x)=A(x,\lambda)\dot{\bp}(x)+(\sigma_0\otimes I_{2n})\bp(x);
\end{equation}
here and below we abbreviate $\sigma_0=\begin{pmatrix}0&0\\1&0\end{pmatrix}$ and recall notations $J_n$ and $J^{(n)}$ in \eqref{eq:newJ} and formula \eqref{newOm}. Computing the scalar product in $\R^{4n}$ of both parts of \eqref{lder} with $J_n\bp$, integrating from $-s$ to $s$, and using the identities 
\begin{align*}
\int_{-s}^s\langle\dot\bp'(x),\, J_n\bp(x)\rangle_{\R^{4n}}\,dx&=
\langle\dot\bp,\, J_n\bp\rangle_{\R^{4n}}\Big|_{-s}^s-\int_{-s}^s\langle\dot\bp(x),\, J_n\bp'(x)\rangle_{\R^{4n}}\,dx\\
&\hskip1cm\text{(integration by parts),}\\
\int_{-s}^s\langle A(x,\lambda)\dot\bp(x),\, J_n\bp(x)\rangle_{\R^{4n}}\,dx&=-
\int_{-s}^s\langle \dot\bp(x),\, J_nA(x,\lambda)\bp(x)\rangle_{\R^{4n}}\,dx\\
&\hskip1cm\text{(formulas $J_n^\top=-J_n$ and \eqref{symm}),}\\
\int_{-s}^s\langle (\sigma_0\otimes I_{2n})\bp(x),\, J_n\bp(x)\rangle_{\R^{4n}}\,dx&=-\int_{-s}^s\langle \big((J\sigma_0)\otimes I_{2n}\big)\bp(x),\,\bp(x)\rangle_{\R^{4n}}\,dx\\&\hskip-1cm =-\int_{-s}^s\|p(x)\|^2_{\R^{2n}}\,dx\quad\text{ (because $\bp=(p,q)^\top$),}
\end{align*}
and $J_n\bp'=JA(x,\lambda)\bp$, we arrive at the equality
\begin{equation}\label{negdef1}\begin{split}
\langle\bp(-s),\,J_n\dot\bp(-s)\rangle_{\R^{4n}}&+
\langle\bp(s),\,(-J_n)\dot\bp(s)\rangle_{\R^{4n}}=\langle\dot\bp,\, J_n\bp\rangle_{\R^{4n}}\Big|_{-s}^s\\ & =- \int_{-s}^s\|p(x)\|^2_{\R^{2n}}\,dx.
\end{split}\end{equation}
A similar argument for $\bw=\bw(x,\lambda)$ yields
\begin{equation}\label{negdef2}
\langle\bw(-s),\,J^{(n)}\dot\bw(-s)\rangle_{\R^{4n}}+
\langle\bw(s),\,(-J^{(n)})\dot\bw(s)\rangle_{\R^{4n}}=0.
\end{equation}
Combining \eqref{negdef1}, \eqref{negdef2} with \eqref{newOm} and 
$$v(\lambda_0)=\Phi_s^{\lambda_0}\big((\bp(\cdot\,,\lambda_0),\bw(\cdot\,,\lambda_0))^\top\big)=\big(\bp(-s,\lambda_0),\bw(-s,\lambda_0),\bp(s,\lambda_0),\bw(s,\lambda_0)\big)^\top$$ we infer 
\begin{align*}
\omega\big(v(\lambda_0),\,\dot v(\lambda_0)\big)&=\big\langle v(\lambda_0),\,\Omega\dot v(\lambda_0)\big\rangle_{\R^{4n}}\\&=\langle\bp(-s),\,J_n\dot\bp(-s)\rangle_{\R^{4n}}+
\langle\bp(s),\,(-J_n)\dot\bp(s)\rangle_{\R^{4n}}\\&\hskip1cm+\langle\bw(-s),\,J^{(n)}\dot\bw(-s)\rangle_{\R^{4n}}+
\langle\bw(s),\,(-J^{(n)})\dot\bw(s)\rangle_{\R^{4n}}\\
&=- \int_{-s}^s\|p(x,\lambda_0)\|^2_{\R^{2n}}\,dx<0,
\end{align*}
thus completing the proof of  \eqref{negdef} 
 and the lemma.
\end{proof}

We will now establish  monotonicity of the Maslov index with respect to the parameter $s$.  The strategy of the proof of the next lemma is similar to the proof of Lemma \ref{lem:cross}. In the lemma we formulate a simple sufficient condition for the crossing form to be sign-definite; however, in the course of its proof we give a general formula \eqref{lastlines}.
We recall that the curve $\Gamma_4$ is parametrized by the parameter $s$ {\em decaying} from $L$ to $s_0$.
\begin{lemma}\label{lem:scross}
For any $\theta\in[0,2\pi]$, any fixed $\lambda\in(0,\lambda_\infty)$, and any $s_0\in(0,L)$, 
each crossing $s_\ast\in(s_0,L)$
of the path $\big\{\Phi_s^\lambda(Y_{s,\lambda})\big\}_{s=s_\ast-\varepsilon}^{s_\ast+\varepsilon}$, with $\varepsilon>0$ small enough, is positive provided the potential $V$ is continuous at the points $\pm s_\ast$ and the matrix 
\begin{equation}\label{Vpd}
\frac12\big(V(-s_\ast)+V(s_\ast)\big) -\lambda I_n \quad\text{ is positive definite.}\end{equation}
 In particular, $B_4=-A_4$ provided $V$ is continuous and positive definite at each point of $[-L,L]$, and $0\notin\Sp(H_\theta)$, $0\notin\Sp(H_{\theta,s_0})$.
\end{lemma}
\begin{proof}
Let $s_\ast\in(s_0,L)$ be a crossing, so that $\Phi_{s_\ast}^{\lambda}(Y_{s_\ast,\lambda})\cap(X\times X)\neq\{0\}$. Let $V^\bot$ be a subspace in $\R^{16n}$ transversal to $\Phi_{s_\ast}^{\lambda}(Y_{s_\ast,\lambda})$.
Then $V^\bot$ is transversal to $\Phi_s^{\lambda}(Y_{s,\lambda})$ for all $s\in[s_\ast-\varepsilon,s_\ast+\varepsilon]$ for $\varepsilon>0$ small enough. Thus, there exists a smooth family of matrices, $\phi(s)$, for $s\in[s_\ast-\varepsilon,s_\ast+\varepsilon]$, viewed as operators $\phi(s): \Phi_{s_\ast}^{\lambda}(Y_{s_\ast,\lambda})\to V^\bot$, such that $\Phi_s^{\lambda}(Y_{s,\lambda})$ is the graph of $\phi(s)$. Fix any nonzero ${\rm v}\in\Phi_{s_\ast}^{\lambda}(Y_{s_\ast,\lambda})\cap(X\times X)$ and consider the curve $v(s)={\rm v}+\phi(s){\rm v}\in\Phi_s^{\lambda}(Y_{s,\lambda})$ for $s\in[s_\ast-\varepsilon, s_\ast+\varepsilon]$ with $v(s_\ast)={\rm v}$. By the definition of $Y_{s,\lambda}$, there is a family of solutions
$(\bp(\cdot\,,s),\bw(\cdot\,,s))^\top$ of \eqref{eq:augodeshort} such that
$v(s)=\Phi_s^{\lambda}\big((\bp(\cdot\,,s),\bw(\cdot\,,s))^\top\big)$. Denoting by dot the derivative with respect to the variable $s$, we claim that 
\begin{equation}\label{snegdef}
\omega\big(v(s_\ast),\,\dot{v}(s_\ast)\big)>0
\end{equation}
provided \eqref{Vpd} holds.
Assuming the claim, we finish the proof as follows: Since for each  nonzero ${\rm v}\in\Phi_{s_\ast}^{\lambda}(Y_{s_\ast,\lambda})\cap(X\times X)$ the crossing form $Q_\cM$ satisfies
\begin{align*}
Q_\cM({\rm v},{\rm v})&=\frac{d}{ds}\Big|_{s=s_\ast}\omega({\rm v},\phi(\lambda){\rm v})=
\frac{d}{ds}\Big|_{s=s_\ast}\omega({\rm v},{\rm v}+\phi(s){\rm v})\\&=\omega\big(v(s_\ast),\,\dot{v}(s_\ast)\big)>0,
\end{align*}
the form is positive definite. Thus, the crossing $s_\ast\in(s_0,L)$ is positive. In particular, taking into account that the path $\Gamma_4=\big\{\Phi^\lambda_s(Y_{s,\lambda})\big\}_{s=L}^{s_0}$ is parametrized by the parameter $s$ {\em decaying} from $L$ to $s_0$, each crossing along $\Gamma_4$ is negative since the assumptions $0\notin\Sp(H_\theta)$, $0\notin\Sp(H_{\theta,s_0})$ and Proposition \ref{cor:mult} imply that all crossings for $\lambda=0$ belong to $(s_0,L)$. Thus, the Maslov index $A_4$ of the path $\Gamma_4$ is equal to $-B_4$.

Starting the proof of claim \eqref{snegdef}, we remark that $s$-derivatives of the solutions $\bp(\cdot,s)$ and $\bw(\cdot,s)$ of \eqref{eq:hillp} and \eqref{eq:rotshort} satisfy the differential equations
\begin{equation}\label{eqpwdot}
\dot{\bp}'(x)=A(x,\lambda)\dot{\bp}(x),\,\,
\dot{\bw}'(x)=\dot{B}(s,\theta){\bw}(x)+B(s,\theta)\dot{\bw}(x),
\end{equation}
where $\dot{B}(s,\theta)$ is computed similarly to \eqref{eq:rotshort}, \eqref{eq:rot} but with $\frac{\mp\theta}{2s}$ replaced by $\frac{\pm\theta}{2s^2}$:
\begin{equation}\label{BdB}
B(s)=-\frac{\theta}{2s}\big(I_{2n}\otimes J\big),\quad
\dot{B}(s)=\frac{\theta}{2s^2}\big(I_{2n}\otimes J\big).
\end{equation}
Clearly, $v(s)=\big(\bp(-s,s),\bw(-s,s),\bp(s,s),\bw(s,s)\big)^\top$
yields
\begin{align*}
\dot{v}(s)&=\big(-\bp'(-s,s)+\dot{\bp}(-s,s), -\bw'(-s,s)+\dot{\bw}(-s,s),\\
&\qquad\qquad \bp'(s,s)+\dot{\bp}(s,s), \bw'(s,s)+\dot{\bw}(s,s)
  \big)^\top.
\end{align*} Using \eqref{eq:omega}, we split the expression for $\omega\big(v(s),\,\dot{v}(s)\big)$ as follows:
\begin{align*}
\langle  v(s),&\Omega\dot{v}(s)\rangle_{\R^{16n}}\\
=&
-\langle \bp(-s,s),\, (J\otimes I_{2n})\bp'(-s,s)\rangle_{\R^{4n}}+
\langle \bp(-s,s),\, (J\otimes I_{2n})\dot{\bp}(-s,s)\rangle_{\R^{4n}}\\&
+\langle \bw(-s,s),\, (J\otimes I_{2n})\bw'(-s,s)\rangle_{\R^{4n}}-
\langle \bw(-s,s),\, (J\otimes I_{2n})\dot{\bw}(-s,s)\rangle_{\R^{4n}}\\&
-\langle \bp(s,s),\, (J\otimes I_{2n})\bp'(s,s)\rangle_{\R^{4n}}-
\langle \bp(s,s),\, (J\otimes I_{2n})\dot{\bp}(s,s)\rangle_{\R^{4n}}\\&
+\langle \bw(s,s),\, (J\otimes I_{2n})\bw'(s,s)\rangle_{\R^{4n}}+
\langle \bw(s,s),\, (J\otimes I_{2n})\dot{\bw}(s,s)\rangle_{\R^{4n}}\\
=&\alpha_1+\alpha_2+\alpha_3+\alpha_4,
\end{align*}
where, using \eqref{eqpwdot} and rearranging terms, the expressions $\alpha_j$ are defined and computed as follows:
\begin{align*}
\alpha_1&=-\langle \bp(-s,s), (J\otimes I_{2n})A(-s,\lambda)\bp(-s,s)\rangle_{\R^{4n}} \\&\hskip2cm-
\langle \bp(s,s), (J\otimes I_{2n})A(s,\lambda){\bp}(s,s)\rangle_{\R^{4n}};\\
\alpha_2&=\langle \bp(-s,s), (J\otimes I_{2n})\dot{\bp}(-s,s)\rangle_{\R^{4n}} -
\langle \bp(s,s), (J\otimes I_{2n})\dot{\bp}(s,s)\rangle_{\R^{4n}}
\\&=
-\int_{-s}^s\frac{d}{dx}\big(\langle \bp(x,s), (J\otimes I_{2n})\dot{\bp}(x,s)\rangle_{\R^{4n}}\big)\,dx
\\&=
-\int_{-s}^s\big(\langle A(x,\lambda) \bp(x,s), (J\otimes I_{2n})\dot{\bp}(x,s)\rangle_{\R^{4n}}\\&\hskip2cm
+\langle \bp(x,s), (J\otimes I_{2n})A(x,\lambda)\dot{\bp}(x,s)\rangle_{\R^{4n}}\big)\,dx \quad\text{(using \eqref{eqpwdot})}
\\&=
-\int_{-s}^s\big(-\langle (J\otimes I_{2n})A(x,\lambda) \bp(x,s), \dot{\bp}(x,s)\rangle_{\R^{4n}}\\&\hskip2cm
+\langle \bp(x,s), (J\otimes I_{2n})A(x,\lambda)\dot{\bp}(x,s)\rangle_{\R^{4n}}\big)\,dx
\\&=0\qquad\text{(using \eqref{symm});}\\
\alpha_3&=\langle \bw(-s,s),\, (J\otimes I_{2n})\bw'(-s,s)\rangle_{\R^{4n}}+
\langle \bw(s,s),\, (J\otimes I_{2n}){\bw}'(s,s)\rangle_{\R^{4n}}\\&=
\langle \bw(-s,s),\,   (J\otimes I_{2n})B(s,\theta)\bw(-s,s)\rangle_{\R^{4n}}
\\&\hskip2cm+
\langle \bw(s,s),\, (J\otimes I_{2n})B(s,\theta){\bw}(s,s)\rangle_{\R^{4n}}
\\&=
\langle \bw(-s,s),\,  -\frac{\theta}{2s} (J\otimes I_{2n})(I_{2n}\otimes J)\bw(-s,s)\rangle_{\R^{4n}}
\\&\hskip2cm+
\langle \bw(s,s),\, -\frac{\theta}{2s} (J\otimes I_{2n})(I_{2n}\otimes J){\bw}(s,s)\rangle_{\R^{4n}}\quad\text{(using \eqref{BdB})}
\\&=
-\frac{\theta}{s}\langle\bw(-s,s), J\otimes(I_n\otimes J)\bw(-s,s)\rangle_{\R^{4n}}
\\&
\quad\text{(since $\Psi_B(x,\theta)$ is orthogonal and commutes with $J\otimes(I_n\otimes J)$)};\\
\alpha_4&=-\langle \bw(-s,s),\, (J\otimes I_{2n})\dot\bw(-s,s)\rangle_{\R^{4n}}+
\langle \bw(s,s),\, (J\otimes I_{2n})\dot{\bw}(s,s)\rangle_{\R^{4n}}\\&=
\int_{-s}^s\frac{d}{dx}\big(\langle \bw(x,s),\,   (J\otimes I_{2n})\dot{\bw}(x,s)\rangle_{\R^{4n}}\big)\,dx\\&=
\int_{-s}^s\big(\langle B(s,\theta)\bw(x,s),\,   (J\otimes I_{2n})\dot{\bw}(x,s)\rangle_{\R^{4n}}\\&\hskip.5cm+\langle \bw(x,s),\,  (J\otimes I_{2n})\big(\dot{B}(s,\theta){\bw}(x,s)+B(s,\theta)\dot\bw(x,s)\big)\rangle_{\R^{4n}}\big)\,dx\,\text{(by \eqref{eqpwdot})}
\\ & = \int_{-s}^s \langle \bw(x,s),   (J\otimes I_{2n}) \dot B(s, \theta) \bw(x,s) \rangle dx \quad \text{(distributing and using \eqref{symm})}
\\&=-\frac{\theta}{2s^2}\int_{-s}^s\langle\bw(x,s), J\otimes(I_n\otimes J)\bw(x,s)\rangle_{\R^{4n}}\,dx
\quad\text{(using \eqref{BdB})}\\
&=\frac{\theta}{s}\langle\bw(-s,s), J\otimes(I_n\otimes J)\bw(-s,s)\rangle_{\R^{4n}}\end{align*}
(since $\Psi_B(x,\theta)$ is orthogonal).
Thus, $\langle  v(s),\,\Omega\dot{v}(s)\rangle_{\R^{16n}}=\alpha_1$. After a short calculation using the condition $\bp(s_\ast,s_\ast)=\big(I_{2n}\otimes 
U(\theta)\big)\bp(-s_\ast,s_\ast)$ (which holds since $s_*$ is a conjugation point), the 
orthogonality of $U(\theta)$, and formulas
\begin{align*}
& \big(J\otimes I_{2n}\big)A(\pm s_\ast,\lambda)=
\big((\lambda I_n-V(\pm s_\ast))\otimes I_2\big)\oplus\big(-I_{2n}\big),\\
&I_{2n}\otimes U(\theta)^{\pm 1}=\big(I_n\otimes U(\theta)^{\pm 1}\big)\oplus\big(I_n\otimes U(\theta)^{\pm 1}\big),\\
&\big(I_{2n}\otimes U(\theta)^{-1}\big)\big(J\otimes I_{2n}\big)A(s_\ast,\lambda)\big(I_{2n}\otimes U(\theta)\big)\\&\hskip3cm=
\big((\lambda I_n-V(s_\ast))\otimes I_2\big)\oplus\big(-I_{2n}\big),
\end{align*}
we conclude that $\omega\big(v(s_\ast),\,\dot{v}(s_\ast)\big)=\alpha_1\Big|_{s=s_\ast}$ is equal to
\begin{align*}
-\big\langle \bp(-s_\ast, s_\ast), 
\Big(\big(2\lambda I_{2n}  -\big(V(-s_\ast)+V(s_\ast)\big)\otimes I_2\big)\oplus\big(-2I_{2n}\big)\Big)\bp(-s_\ast,s_\ast)\big\rangle_{\R^{4n}}.\end{align*}
Since $\bp(-s_\ast,s_\ast)=\big(p(-s_\ast,s_\ast), q(-s_\ast,s_\ast)\big)^\top$, we therefore have the following final formula for the crossing form:
\begin{align}\omega\big(v(s_\ast),\,\dot{v}(s_\ast)\big)&
=2\big\langle p(-s_\ast,s_\ast),
\big(\frac12\big(V(-s_\ast)+V(s_\ast)\big)\otimes I_2-\lambda I_{2n}\big)p(-s_\ast,s_\ast)\big\rangle_{\R^{2n}}\nonumber\\&\hskip2cm
+2 \|q(-s_\ast,s_\ast)\|_{\R^{2n}}^2.\label{lastlines}
\end{align}
In particular, \eqref{Vpd} implies \eqref{snegdef}.
\end{proof}

We will prove next a version of Lemma \ref{lem:llss} (ii) for $\theta=0$ or $\theta=2\pi$. It is interesting to note that although the conclusion of the next lemma concerns the spectrum of the operators $H_{0,s}$, its proof uses topological arguments which led to Corollary \ref{cor3.9}. We recall the notation $\mo(\cH)=\dim(\ran\,\cP)$ for the Morse index of an invertible selfadjoint semi-bounded from above operator $\cH$; here, 
\begin{equation}\label{DefRP}
\cP=(2\pi i)^{-1}\int_\gamma(z-\cH)^{-1}\,dz
\end{equation} is the Riesz projection corresponding to the positive part $\Sp(\cH)\cap(0,+\infty)$ of the spectrum of $\cH$, and $\gamma$ is a smooth curve enclosing this part of the spectrum.

\begin{lemma}\label{llssnew}
Assume that $\theta=0$ or $\theta=2\pi$ and that the potential $V$ is continuous at $x=0$ and the matrix $V(0)$ is invertible. If $\lambda_\infty>\|V\|_\infty$ and $s_0\in(0,L]$ is sufficiently small then $0\notin\Sp(H_{0,s_0})$ and $B_1=\mo(V(0))$; in particular, if $V(0)$ is negative definite then $B_1=0$.
\end{lemma}
\begin{proof}
Since $H_{0,s}=H_{2\pi,s}$ because the boundary conditions \eqref{sthetaBC} are the same for $\theta=0$ and $\theta=2\pi$, and taking into account Proposition \ref{cor:mult}, we will consider only the case $\theta=0$.
If $\theta=0$ and $s>0$ then $H_{0,s}$ is the operator in $L^2([-s,s])$ defined by $(H_{0,s}y)(x)=y''(x)+V(x)y(x)$, $|x|\le s$, with the domain
\begin{equation*}\begin{split}
\dom(H_{0,s})&=\big\{y\in L^2([-s,s])\big|\, y,y'\in AC_{\loc}([-s,s]), y''  \in L^2([-s,s])\text{ and the} \\
&\qquad\text{periodic boundary conditions $y(s)=y(-s)$, $y'(s)=y'(-s)$ hold}\big\}.\end{split}
\end{equation*}
It is convenient to ``rescale'' the operator $H_{0,s}$ to $L^2([-L,L])$ by introducing the operator $H_0(s)$ in $L^2([-L,L])$
defined by \begin{equation*}
(H_0(s)y)(x)=\big({L}/{s}\big)^2y''(x)+V\big(sx/L\big)y(x), \quad |x|\le L, \end{equation*}
with the domain
\begin{equation*}\begin{split}
\dom(H_0(s))&=\big\{y\in L^2([-L,L])\big|\, y,y'\in AC_{\loc}([-L,L]), y''  \in L^2([-L,L])\text{ and the} \\
&\text{periodic boundary conditions $y(L)=y(-L)$, $y'(L)=y'(-L)$ hold}\big\}.\end{split}
\end{equation*}
Writing the eigenvalue equation $(H_{0,s}y)(x)=\lambda y(x)$, $|x|\le s$, at the point $x=s\widehat{x}/L$ for $|\widehat{x}|\le L$, introducing $z(\widehat{x})=y(s\widehat{x}/L)$, and passing to the eigenvalue equation $(H_0(s)z)(\widehat{x})=\lambda z(\widehat{x})$, $|\widehat{x}|\le L$,  we observe that 
\begin{equation}\label{speceq}
\Sp(H_{0,s}; L^2([-s,s]))=\Sp(H_0(s); L^2([-L,L]))\quad\text{for all $s\in(0,L]$}.\end{equation}
In addition to $H_0(s)$, we introduce a constant coefficient operator  $H_0^{(0)}(s)$ on $L^2([-L,L])$
defined by $(H_0^{(0)}(s)y)(x)=\big({L}/{s}\big)^2y''(x)+V(0)y(x)$, $|x|\le L$, with the domain $\dom H_0^{(0)}(s)=\dom H_0(s)$. Since 
\begin{equation}\label{limVsL}
\|H_0(s)-H_0^{(0)}(s)\|_{\cB(L^2([-L,L]))}=\sup_{|x|\le L}\|V(sx/L)-V(0)\|\to0\text{ as $s\to0$}
\end{equation} by the continuity assumption in the lemma, we can use Theorem \ref{KatoTh} to conclude that 
\begin{equation}\label{limsp}
\dist\big(\Sp(H_0(s)), \Sp(H_0^{(0)}(s))\big)\to0\text{ as $s\to0$}.
\end{equation}
Since the operator $H_0^{(0)}(s)$ is a constant coefficient operator with periodic boundary conditions, passing to  the Fourier series $y(x)=\sum_{k\in\bbZ}y_ke^{i\pi kx/L}$, $|x|\le L$, we calculate:
\begin{equation}\label{spH00}
\Sp(H_0^{(0)}(s))=\bigcup_{k\in\bbZ}\Big(-\big(\pi k/s\big)^2+\Sp(V(0))\Big).
\end{equation}
Let $\nu_j$ denote the eigenvalues of the matrix $V(0)$ and let $\varkappa=\mo(V(0))$ denote the number of the positive eigenvalues counting multiplicities. Since $0\notin\Sp(V(0))$ by the assumption, we can find a $\delta>0$, and enumerate the eigenvalues in $\Sp(V(0))$ such that
\[ - \|V(0)\|\le\dots \le \nu_{-1}<-\delta<0<\delta<\nu_1\le\dots\le\nu_\varkappa\le\|V(0)\|.\]
Choose $s_1\in(0,L)$ so small that $\|V(0)\|+\delta<(\pi/s_1)^2$, see Figure \ref{Ftheta0}. 
\begin{figure}
\begin{picture}(100,100)(-20,0)
\put(2,2){0}
\put(10,5){\vector(0,1){95}}
\put(5,10){\vector(1,0){95}}
\put(71,40){\text{\tiny $\Gamma_2$}}
\put(12,40){\text{\tiny $\Gamma_4$}}
\put(45,75){\text{\tiny $\Gamma_3$}}
\put(43,14){\text{\tiny $\Gamma_1$}}
\put(100,12){$\lambda$}
\put(12,100){$s$}
\put(80,20){\line(0,1){60}}
\put(80,8){\line(0,1){4}}
\put(78,0){$\lambda_\infty$}
\put(10,20){\line(1,0){70}}
\put(10,80){\line(1,0){70}}
\put(0,78){$L$}
\put(-2,18){$s_0$}
\put(8,20){\line(1,0){4}}
\put(-2,38){$s_1$}
\put(8,38){\line(1,0){4}}
\put(-2,28){$s_4$}
\put(8,28){\line(1,0){4}}
\put(40,20){\circle*{4}}
\put(10,50){\circle*{4}}
\put(10,70){\circle*{4}}
\put(20,80){\circle*{4}}
\put(40,80){\circle*{4}}
\put(60,80){\circle*{4}}
\put(20,87){{\tiny \text{$\theta$-eigenvalues}}}
\put(17,24){{\tiny \text{$(\theta,s_0)$-eigenvalues}}}
\put(-10,20){\rotatebox{90}{{\tiny conjugate points}}}
\put(90,20){\rotatebox{90}{{\tiny no conjugate points}}}
\end{picture}
\caption{$\theta=0$  and the numbers $s_1>s_2>s_3>s_4\ge s_0>0$ in the proof of Lemma \ref{llssnew} are small enough}\label{Ftheta0}
\end{figure} 
Then, for each $s\in(0,s_1]$, the eigenvalues $\nu_j-(\pi k/s)^2$ of the operator $H_0^{(0)}(s)$ are positioned as follows:
\begin{align*}
0&<\delta<\nu_1\le\dots\le\nu_\varkappa, \text{ for $j\ge1$ and $k=0$,}\\
\dots&\le\nu_1-(k\pi/s)^2\le\dots\le\nu_\varkappa-(k\pi/s)^2
\le\dots
\le\nu_1-(\pi/s)^2\le\dots\le\nu_\varkappa-(\pi/s)^2\\&\quad\le\|V(0)\|-(\pi/s)^2\le\|V(0)\|-(\pi/s_1)^2<-\delta<0, \text{ for $j\ge1$ and $k\in\bbZ\setminus\{0\}$},\\  \nu_j&-(k\pi/s)^2<-\delta,\text{ for $j\le-1$ and $k\in\bbZ$}.
\end{align*}
In particular, for each $s\in(0,s_1]$ we have $\Sp(H_0^{(0)}(s))\cap[-\delta,\delta]=\emptyset$, 
\begin{equation}\label{deltsin}
\Sp(H_0^{(0)}(s))\cap(0,+\infty)=\{\nu_1,\dots,\nu_\varkappa\}\subset\big(\delta,\|V(0)\|\big),
\end{equation} and $\mo(H_0^{(0)}(s))=\mo(V(0))$. Using \eqref{limsp}, \eqref{deltsin} one can find a sufficiently small $s_2\in(0,s_1)$ such that for all $s\in(0,s_2]$ one has:
\begin{equation}\label{Idc}
0\notin\Sp(H_0(s))\text{ and } \Sp(H_0(s))\cap(0,+\infty)\subset \big(\delta/2,\|V(0)\|+\delta/2\big).
\end{equation}
By \eqref{speceq}, conclusions \eqref{Idc}  hold with $\Sp(H_0(s))$ replaced by $\Sp(H_{0,s})$.

Fix any $s_3\in(0,s_2)$. We now claim that
\begin{equation}\label{clHmor}
\sup_{s\in(0,s_3]}\mo (H_0(s))<\infty.
\end{equation}
Postponing the proof of claim \eqref{clHmor}, we proceed with the proof of the lemma. 

Let $P_s$, respectively, $P_s^{(0)}$ denote the Riesz projection for the operator $H_0(s)$, respectively, $H_0^{(0)}(s)$ that corresponds to the positive part of its spectrum. The Riesz projection is defined by formula \eqref{DefRP} with $\cH$ replaced by $H_0(s)$, respectively, $H_0^{(0)}(s)$, and with $\gamma$ being the rectangle in the complex plane with the vertices at the points $\pm i\delta$ and $\|V(0)\|+\delta\pm i\delta$. Inclusions \eqref{deltsin}, \eqref{Idc} imply, for $s\in(0,s_2]$,
\begin{equation}\label{distsp}
\dist\big(\Sp(H_0^{(0)}(s)),\gamma\big)\ge\delta,\, \dist\big(\Sp(H_0(s)),\gamma\big)\ge\delta/2.
\end{equation}
Using \eqref{distsp} and that $H_0(s)$ is selfadjoint, for $z\in\gamma$ we arrive at the estimate \[\|\big(z-H_0(s)\big)^{-1}\|=1/\dist(\Sp(H_0(s)),z)\le1/\dist(\Sp(H_0(s)),\gamma)\le2/\delta,\]
and a similar estimate for $\|\big(z-H_0(s)\big)^{-1}\|$.
Using \eqref{limVsL} and 
\begin{equation*}
P_s-P_s^{(0)}=(2\pi i)^{-1}\int_\gamma\Big(\big(z-H_0(s)\big)^{-1}
\big(H_0(s)-H_0^{(0)}(s)\big)\big(z-H_0^{(0)}(s)\big)^{-1}
\Big)\,dz,
\end{equation*}
we conclude that
\begin{equation}\label{limPPs}
\big\|P_s-P_s^{(0)}\big\|_{\cB(L^2([-L,L]))}\to0\text{ as $s\to0$}.\end{equation}
Noting that $\mo (H_0^{(0)}(s))=\dim(\ran P_s^{(0)})=\tr (P_s^{(0)})=\mo(V(0))$ by \eqref{deltsin} and that the dimensions $\mo (H_0(s))=\dim(\ran P_s)=\tr (P_s)$ are bounded uniformly for $s\in(0,s_3]$ by \eqref{clHmor}, we introduce the subspace $R_s=\ran P_s^{(0)}\oplus\ran P_s$ and denote $R=\sup_{s\in(0,s_3]}\dim R_s<\infty$. Viewing $P_s-P_s^{(0)}$ as a finite dimensional operator in $R_s$, we infer, using a simple inequality relating trace and norm:
\begin{equation}\label{tracein}\begin{split}
\big|\tr(P_s)-\tr(P_s^{(0)})\big|&=\big|\tr(P_s-P_s^{(0)})\big|\\&\le R\big\|P_s-P_s^{(0)}\big\|_{\cB(R_s)}\le R\big\|P_s-P_s^{(0)}\big\|_{\cB(L^2([-L,L]))}.\end{split}
\end{equation} 
We recall the $0\notin\Sp(H_{0,s})$ and $\mo(H_{0,s})=\mo(H_0(s))$ by \eqref{Idc} and \eqref{speceq}.
Thus, using Proposition \ref{cor:mult}, for any $s_0\in(0,s_3]$, the number $B_1=\mo(H_{0,s_0})$ of the crossings along $\Gamma_1$ is equal to the Morse index $\mo(H_0(s_0))=\tr(P_{s_0})$.  In order to establish the required in the lemma equality $B_1=\mo(V(0))$, it suffices to show that $\tr(P_{s_0})=\tr(P_{s_0}^{(0)})$ for all small enough $s_0$. Indeed, the latter equality implies
\begin{equation}\label{fineq}
B_1=\mo(H_{0,s_0})=\mo(H_0(s_0))=\tr(P_{s_0})=\tr(P_{s_0}^{(0)})=\mo(V(0)),
\end{equation}
as needed in the lemma.  
Since the functions $s\mapsto\tr(P_s)$, $s\mapsto\tr(P_s^{(0)})$ take integer values, it suffices to show the existence of a small 
$s_4\in(0,s_3)$ such that the right-hand side of \eqref{tracein} is smaller than $1$ for all $s\in(0,s_4]$. But this follows from \eqref{limPPs}, thus concluding the proof of the lemma.

It remains to prove claim \eqref{clHmor}. This is the part of the proof  based on Corollary \ref{cor3.9}. Since $0\notin\Sp(H_{0,s})$ and $\mo(H_{0,s})=\mo(H_0(s))$ for all $s\in(0,s_2]$ by \eqref{Idc} and \eqref{speceq}, in order to show \eqref{clHmor} it suffices to check that the number of crossings $B_1=\mo(H_{0,s_0})$ along the curve $\Gamma_1=\Gamma_{1,(0,s_0)}$ is estimated from above by a finite number that does not depend on $s_0\in(0,s_3]$ (we recall that  $s_3\in(0,s_2))$. Take any $s_0\in(0,s_3]$ and construct the curve $\Gamma=\Gamma_{1,(0,s_0)}\cup\Gamma_2\cup\Gamma_3\cup\Gamma_{4,(0,s_0)}$ as described in Remark \ref{remN3.9}. First, we remark that due to Proposition \ref{cor:mult} there are no crossings of the portion of the curve $\Gamma_{4,(0,s_0)}$ given by $\{\Phi_{s}^0(Y_{s,0}))\big| s\in[s_0,s_2]\}$ since $0\notin\Sp(H_{0,s})$ for all $s\in[s_0,s_2]$.
Second, we remark that with no loss of generality we may assume that the curve $\Gamma_{4,(0,s_0)}\cup\Gamma_3$ is regular. (Indeed, otherwise, use  homotopy with the fixed endpoints $\Phi_{s_2}^0(Y_{s_2,0})$ and $\Phi_L^{\lambda_\infty}(Y_{L,\lambda_\infty})$ of the portion of this curve given by $\{\Phi_{s}^0(Y_{s,0}))\big| s\in[s_2,L]\}\cup\Gamma_3$ to transform it into a regular curve.) Since the regular crossings are  isolated, based on the two remarks just made we conclude that the number $B_4$ of the crossings of $\Gamma_{4,(0,s_0)}$ is finite and does not depend on $s_0$. Clearly, $B_3$ is finite and does not depend on $s_0$ either. By Corollary \ref{cor3.9} we know that $0=A_1+A_2+A_3+A_4$. By Lemma \ref{lem:llss} (i) we have $A_2=0$. By Lemma \ref{lem:cross} we know that $A_3=B_3$ and $A_1=-B_1$. Combining all this, we have the required estimate
\begin{align*}
\mo(H_{0,s_0})&=B_1=-A_1=A_3+A_4=B_3+A_4\le B_3+|A_4|\le B_3+B_4,
\end{align*}
which concludes the proof of claim \eqref{clHmor} and the lemma.\end{proof}

We will now summarize the count of eigenvalues and conjugate points via the Morse and Maslov indices. Recall that the number of positive eigenvalues of a self-adjoint  operator (counting their multiplicities) is called its {\em Morse index}, and is denoted by $\mo(\cdot)$. 
Also, recall definition \eqref{DefMI} of the Maslov index. The Maslov index of $\Gamma_4=\Gamma_{4,(\theta,s_0)}$ is the number $A_4=A_{4,(\theta,s_0)}$ which can also be thought of as the {\em Maslov index} of equation (\ref{eq:hill}). Also, we recall definition \eqref{defBi} of the numbers $B_i$, and note that the expressions for $A_i$ and $B_i$ do not contain the first and the last terms provided the endpoints of $\Gamma_i$ are {\em not} crossings. 
In this case, we can interpret $B_3$ and $B_4$ in terms of the eigenvalues of $H_\theta$ and the conjugate points.  
\begin{theorem} \label{th:mas2}  Let us fix $\theta\in[0,2\pi]$, and let the numbers $A_i=A_{i,(\theta,s_0)}$ and $B_i=B_{i,(\theta,s_0)}$ be defined in \eqref{defAi} and \eqref{defBi} for a (small) $s_0>0$ and a (large) $\lambda_\infty>0$. Then the following assertions hold.
\begin{enumerate}
\item[(i)] The Maslov index of the curve $\Gamma$ is zero for any $s_0\in(0,L)$ and $\lambda_\infty>0$.
\item[(ii)] If $0\notin\Sp(H_\theta)$ then the Maslov index $A_3$ of the curve $\Gamma_3$ satisfies $A_3=B_3$.
\item[(iii)] If $0\notin\Sp(H_\theta)$ and $\lambda_\infty$ is large enough then $B_3$ is equal to the number of the positive  $\lambda$ for which there exists a solution to the original boundary value problem for equation (\ref{eq:hillsys}) on $[-L,L]$ subject to the boundary condition  (\ref{eq:pereigen}) counting multiplicities, that is, to  the {\em Morse index}  of the operator $H_\theta$ in $L^2([-L,L])$:
 $$B_3=\mo\big(H_\theta\big).$$
\item[(iv)] If $0\notin\Sp(H_\theta)$ and $0\notin\Sp(H_{\theta,s_0})$ for some $s_0>0$ then $B_4$ is equal to the number of the conjugate points for  $\lambda =0$ counting multiplicities, that is, the number of such $s\in(s_0,L)$ for which there exists a nontrivial solution to the boundary value problem for equation (\ref{eq:hillsys}) on $[-s,s]$ subject to the boundary condition \eqref{eq3.28}.
\item[(v)] If $\theta\in(0,2\pi)$, $\lambda_\infty>0$ is large enough and $s_0>0$ is small enough then $A_3 = -A_4$. If, in addition,  $0\notin\Sp(H_\theta)$ then the Maslov index and the Morse index are related as follows:
\begin{equation}\mo(H_\theta)=-\mi(\Gamma_4, X\times X).\end{equation}
\item[(vi)] If the potential $V$ is continuous and positive definite on $[-L,L]$, and the assumptions in $(iv)$ hold then $A_4=-B_4$. If, in addition, the assumptions in $(v)$ hold then the Morse index can be computed as follows:
\begin{equation}\mo(H_\theta)=B_4.\end{equation}
\item[(vii)] If $\theta=0$ or $\theta=2\pi$, $\lambda_\infty>0$ is large enough and $s_0>0$ is small enough, the potential $V$ is continuous at the point $x=0$, and $0\notin\Sp(V(0))$ then $A_1=-B_1=\mo(V(0))$. If, in addition, $0\notin\Sp(H_\theta)$ then 
\begin{equation}\label{cEVn}
\mo(H_\theta)=-\mi(\Gamma_4, X\times X)+\mo(V(0)).
\end{equation}
\end{enumerate}
Finally, if $0$ is not in the spectrum of the operator $H$ in $L^2(\bbR)$, then, for all $\theta\in[0,2\pi]$, the Morse index of $H_\theta$ does not depend on $\theta$, and is greater than or is equal to the number of disjoint spectral bands of $H$ in the (unstable) right half-line, and is  equal to the number of the spectral bands of $H$ if they are disjoint. 
\end{theorem}
\begin{proof}
Assertion (i) is proved in Corollary \ref{cor3.9}. Assertion (ii) follows from Lemma \ref{lem:cross}, while (iii) and (iv) are proved in Proposition \ref{cor:mult}. Assertion (v) follows from (i) and Lemma \ref{lem:llss} while (vi) follows from Lemma \ref{lem:scross}. Assertion (vii) follows from Lemma \ref{llssnew}.
\end{proof}

\section{The Mathieu equation: an example}\label{sec:me} Now we will use a well known Mathieu equation, see, e.g. \cite{BeOr78, JS99},
 as an illustration of the phenomena described in Theorem \ref{th:mas2}. This example will also give some indication as to how to handle the loss of regularity of crossings when $\theta=0$ or $\theta=2\pi$ and the curve $\Gamma_{(0,0)}$ is constructed as in Corollary \ref{cor:3.10}.
 Specifically, let us consider the Mathieu equation
\begin{equation} \label{eq:mathieu}
y'' + 3.2 \cos(2 x)y = \lambda y, \, x\in [-\pi,\pi],
\end{equation}
where we have chosen the usual parameter in the equation to be $-1.6$ purely for convenience and choose $L = \pi$ (we could of course also choose $L = k\pi/2$, $k \in \Z$).

 Letting $\Psi^\bbC_A(x,\lambda)$ be the fundamental solution matrix to the $(2\times2)$ first order (complex) system associated with \eqref{eq:mathieu} such that $\Psi^\bbC_A(-\pi,\lambda)=I_2$, we have then that the propagator for all $s \in [0, \pi]$ is given by $M^\bbC_A(s,\lambda) = \Psi^\bbC_A(s,\lambda) \big(\Psi^\bbC_A(-s,\lambda)\big)^{-1}$. Now for a fixed value of $\theta\in[0,2\pi]$ we can look for $\theta$-eigenvalues and conjugate points of \eqref{eq:mathieu}. 
 A $\theta$-eigenvalue will occur when $M^\bbC_A(\pi,\lambda)$ has an eigenvalue $e^{i \theta}$, while a conjugate point will be a value of $s$ such that $M^\bbC_A(s,0)$ has an eigenvalue $e^{i \theta}$. That is, the following two quantities can be computed: 
\begin{align*}
 B_3&=\{ \textrm{The number of  } \lambda \in [0,\infty) \textrm{ such that } \det \left( M^\bbC_A(\pi, \lambda) - e^{i\theta}I_2 \right) = 0 \}, \\
B_4&= \{ \textrm{The number of  } s \in [0,\pi] \textrm{ such that } \det \left( M^\bbC_A(s, 0) - e^{i\theta}I_2 \right) = 0 \}.
\end{align*}
The number $B_4$ here corresponds to the choice $s_0=0$.
The graphs in Figures  \ref{fig:eigenvalues} and \ref{fig:conjugate} were computed using {\em Mathematica}'s numerical Mathieu equations, and plot the values of the $\theta$-eigenvalues and the conjugate points versus values of $\theta \in [0,2 \pi]$. 

\begin{figure}[h]
\includegraphics[scale=1]{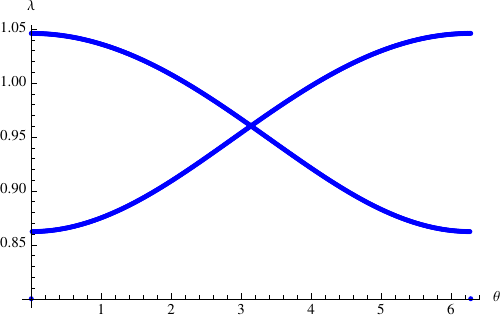}
\caption{A plot of the location of the $\theta$-eigenvalues versus $\theta$
in the Mathieu example (where $s=\pi$). It is clear that at $\theta = \pi$ there is a double eigenvalue.} \label{fig:eigenvalues}
\end{figure}

\begin{figure}[h]
\includegraphics[scale=1]{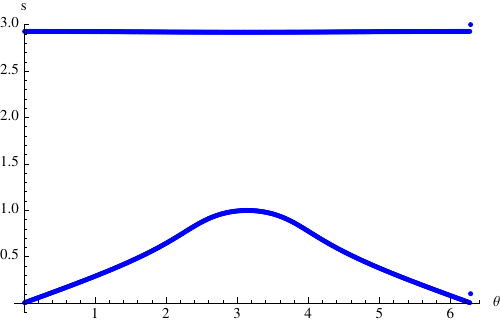}
\caption{A plot of the location of the conjugate points $s$ versus $\theta$ in the Mathieu example (where $\lambda=0$). } \label{fig:conjugate}
\end{figure}

The graphs show that in our numerical experiments the quantities $B_3$ and $B_4$ are equal for all $\theta\in[0,2\pi]$. 
It is worth noting that the multiplicity of the eigenvalue $\lambda$ when $\theta = \pi$ is two, however this is `canceled' out by two crossings along $\Gamma_4$ - i.e. we have two separate conjugate points, each with multiplicity one, and thus our calculations are in concert with Theorems \ref{th:mas1} and \ref{th:mas2}.

Theorem \ref{th:mas2} (ii), (v) tells us that for any $\theta\in(0,2\pi)$ the number of $\theta$-eigenvalues for $s=\pi$ will be the same as the (signed) count of the number of conjugate points for $\lambda = 0$, that is, that $B_3=A_3=-A_4$ as soon as we chose $\lambda_\infty>0$ large enough and $s_0>0$ small enough. We now need to choose a small $s_0>0$ as the arguments in Theorem \ref{th:mas2} (v) involve the curve $\Gamma_1$ as defined in Remark \ref{remN3.9}.

We recall that $B_3=A_3$ due to Lemma \ref{lem:cross} for any $\theta\in[0,2\pi]$. Also, for the chosen value 3.2 of the parameter in the Mathieu equation it is known that $\lambda=0$ is not a $\theta$-eigenvalue for any $\theta\in[0,2\pi]$, see \cite{BeOr78, JS99}.

Lemma \ref{lem:scross} can be applied for the crossings at the conjugate points $s_\ast\in(0,\pi)$ such that $3.2 \cos 2 s_\ast>0$. For any $\theta\in[0,2\pi]$ the latter inequality certainly holds for the upper conjugate point in Figure \ref{fig:conjugate}, and thus the crossing form is positive at this crossing by Lemma \ref{lem:scross}. At the lower conjugate point in Figure \ref{fig:conjugate} the inequality $3.2 \cos 2 s_\ast>0$ does not hold for $\theta$ close to $\pi$, and thus one can not use the sufficient condition \eqref{Vpd} for the crossing form to be positive formulated in Lemma \ref{lem:scross}. However, using the explicit formula for the crossing form in equation \eqref{lastlines}, we computed directly the value of the crossing form at each crossing.  As can be seen from Figure \ref{fig:potential}, the crossing form is positive also for the lower conjugate point.
This implies that $B_4=-A_4$ for the crossings in $(0,\pi)$, and confirms the equality $B_3=A_3=B_4=-A_4$ in yet another way and for all $\theta\in(0,2\pi)$.

\begin{figure}[h]
\includegraphics[scale=1]{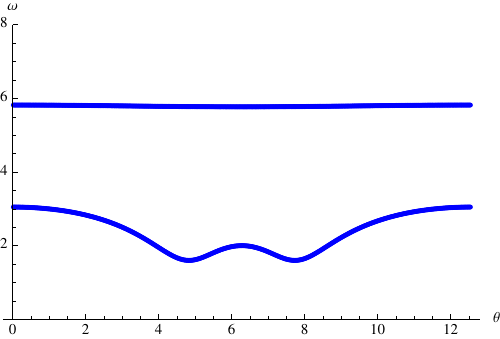}
\caption{A plot of the value $\omega=\omega(v(s_\ast(\theta)), \dot{v}(s_\ast(\theta)))$ of the crossing form versus $\theta$ at the two conjugate points $s_\ast=s_\ast(\theta)$. The crossing form was computed using the right-hand side of equation \eqref{lastlines}, and is positive.}\label{fig:potential}
\end{figure}


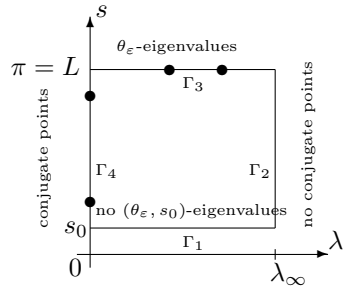
\begin{figure}
\begin{picture}(100,100)(-20,0)
\put(2,2){0}
\put(10,5){\vector(0,1){95}}
\put(5,10){\vector(1,0){95}}
\put(70,40){\text{\tiny $\Gamma_2$}}
\put(12,40){\text{\tiny $\Gamma_4$}}
\put(45,73){\text{\tiny $\Gamma_3$}}
\put(45,13){\text{\tiny $\Gamma_1$}}
\put(100,12){$\lambda$}
\put(12,100){$s$}
\put(80,20){\line(0,1){60}}
\put(80,8){\line(0,1){4}}
\put(78,0){$\lambda_\infty$}
\put(10,20){\line(1,0){70}}
\put(10,80){\line(1,0){70}}
\put(-20,78){$\pi= L$}
\put(0,18){$s_0$}
\put(10,30){\circle*{4}}
\put(10,70){\circle*{4}}
\put(40,80){\circle*{4}}
\put(60,80){\circle*{4}}
\put(20,87){{\tiny \text{$\theta_{\varepsilon}$-eigenvalues}}}
\put(12,25){{\tiny \text{no $(\theta_{\varepsilon},s_0)$-eigenvalues}}}
\put(-10,20){\rotatebox{90}{{\tiny conjugate points}}}
\put(90,20){\rotatebox{90}{{\tiny no conjugate points}}}
\end{picture}
\caption{ We let $\theta_\varepsilon \to 0$ and choose $s_0=s_0(\theta_\varepsilon)>0$ sufficiently small. Although $s_0^{\min}(\theta_\varepsilon) \to 0$,  for each $\theta_\varepsilon$ we still have that  $\lambda=0$ is  not a $\theta_{\varepsilon}$-eigenvalue, $\lambda=0$ is  not a $(\theta_{\varepsilon},s_0)$-eigenvalue, and conclusions of Theorems \ref{th:mas1} and \ref{th:mas2} still hold.}\label{F3}
\end{figure}

The case as $\theta \to 0$ or $\theta\to2\pi$ poses more of a problem. As these two possibilities are analogous,  we consider, as usual, the case $\theta \to 0$. In this case the hypothesis of Lemma  \ref{lem:llss} (ii) is not satisfied, so we can not expect to have a non-zero lower bound $\lim_{\theta \to 0^+} s_0^{\min}(\theta)$, where $s_0^{\min}(\theta)$ is defined as follows:
\begin{equation*}\begin{split}
s_0^{\min}(\theta):=\inf\big\{s\in(0,L]\big|\, &\text{ for some $\lambda>0$ on $[-s,s]$ there exists a nonzero solution}\\
&\qquad\text{  of the boundary value problem \eqref{eq:hill}, \eqref{sthetaBC} } \big\}.\end{split}
\end{equation*} Indeed, as seen from the plot, $\lim_{\theta \to 0^+} s_0^{\min}(\theta) = 0$ in the Mathieu example (see Figure \ref{fig:conjugate}).
One can choose, however, a sequence of nonzero $\theta_\varepsilon$ that converges to zero (see Figure \ref{F3}).
It is worth noting that the $\theta_\varepsilon$-eigenvalues stabilize away from zero, even though the lower bound $s_0^{\min}(\theta_\varepsilon)$ tends to zero. This is because we have chosen the parameter $q$ in the Mathieu equation $q = -1.6$ so that $0$ was not an eigenvalue for any $\theta$ (and in particular for $\theta = 0$). 
We could have similarly perturbed $\theta_\varepsilon$ away from zero in the negative direction, and we observe the same behavior. The numerical calculations can be  summarized as follows: For small $\theta_{\pm \varepsilon} $, the values of the $\theta$-eigenvalues are 0.862 and 1.046, each of multiplicity one, while the conjugate points are a small positive number and 2.926, each of multiplicity one.

 It is also worth noting that even though the entire boundary curve $\Gamma$ when $\theta=0$ has to be defined for $s_0=0$ as in Corollary \ref{cor:3.10} because \eqref{eq:const} is not defined at $s=0$, the curves $\Gamma_3$ and $\Gamma_4$ are regular. 
 
 Alternatively, if $\theta=0$, we can pick a small $s_0>0$ and define
 $\Gamma_1$ and $\Gamma_4$ as described in Remark \ref{remN3.9}. In this case, the curve $\Gamma_4$ contains only one conjugate point
 (the upper conjugate point $s_\ast$ on the vertical line $\theta=0$, see  Figure \ref{fig:conjugate}). Thus, $B_4=1$. Since $\cos s_\ast>0$ for the upper conjugate point, by Lemma \ref{lem:scross} we have $A_4=-B_4=-1$. Since $\cos 0>0$, we have $\mo(V(0))=1$, thus confirming the count in \eqref{cEVn} since $B_3=2$ is the  number of $\theta$-eigenvalues when $\theta=0$.

\bibliographystyle{amsalpha}

\begin{thebibliography}{XXXX}


\bibitem[A01]{A01} A. Abbondandolo, 
{\em Morse Theory for Hamiltonian Systems.} 
Chapman \& Hall/CRC Res. Notes Math. {\bf 425}, Chapman \& Hall/CRC, Boca Raton, FL, 2001.


\bibitem[Ar67]{arnold67}  V. I. Arnold,
{\em Characteristic classes entering in quantization conditions,}
Func. Anal. Appl. {\bf 1} (1967), 1--14.

\bibitem[Ar85]{Arn85} V. I.  Arnold,
{\em Sturm theorems and symplectic geometry,}
 Func. Anal. Appl. {\bf 19} (1985), 1--10.

\bibitem[BO78]{BeOr78} C. \ Bender and S. \ Orszag, {\em Advanced  Mathematical Methods for Scientists and Engineers,} McGraw-Hill, Sydney, 1978.

\bibitem[B56]{B56} R.\ Bott, 
{\em On the iteration of closed geodesics and the Sturm intersection theory,} 
Comm. Pure Appl. Math. {\bf 9} (1956), 171--206.

\bibitem[CDB06]{CDB06} F.\ Chardard, F.\ Dias and T.\ J.\ Bridges, {\em Fast computation of the Maslov index for hyperbolic linear systems with periodic coefficients,} J. Phys. A {\bf 39} (2006), 14545--14557.


\bibitem[CDB09]{CDB09} F.\ Chardard, F.\ Dias and T.\ J.\ Bridges, {\em
 Computing the Maslov index of solitary waves. I. Hamiltonian systems on a four-dimensional phase space,} Phys. D {\bf 238} (2009), 1841--1867.
 
 \bibitem[CDB11]{CDB11} F.\ Chardard, F.\ Dias and T.\ J.\ Bridges, {\em Computing the Maslov index of solitary waves, Part 2: Phase space with dimension greater than four,} Phys. D {\bf 240} (2011), 1334--1344.

\bibitem[CZ84]{CZ84} C. Conley and E. Zehnder, {\em Morse-type index theory for flows and periodic solutions for Hamiltonian equations,} Comm. Pure Appl. Math. {\bf 37} (1984), 207--253. 

\bibitem[DJ11]{DJ11} J.\ Deng and C.\ Jones, {\em Multi-dimensional Morse Index Theorems and a symplectic view of elliptic boundary value problems,}
Trans. Amer. Math. Soc. {\bf 363} (2011), 1487--1508.

\bibitem[D76]{D76} J. J. Duistermaat, {\em On the Morse index in variational calculus,} Advances in Math. {\bf 21} (1976), 173--195.

\bibitem[FJN03]{FJN} R.\ Fabbri, R.\ Johnson and C.\ N\'u\~nez,
{\em Rotation number for non-autonomous linear Hamiltonian
systems I: Basic properties,}
Z. angew. Math. Phys. {\bf 54} (2003), 484--502.

\bibitem[F04]{F} K.\ Furutani, {\em Fredholm-Lagrangian-Grassmannian and the Maslov index,} Journal of Geometry and Physics {\bf 51} (2004), 269--331.
        

\bibitem[Ga93]{Ga93} R.\ A.\ Gardner, 
{\em On the structure of the spectra of periodic travelling waves,}
 J. Math. Pures Appl. {\bf 72} (1993) 415--439.
 
 \bibitem[G07]{G07} F.\ Gesztesy,
{\it Inverse spectral theory as influenced by Barry Simon}, In: {\em Spectral Theory and Mathematical Physics: a Festschrift in Honor of Barry Simon's 60th Birthday}, pp.\ 741 -- 820,
Proc. Sympos. Pure Math. {\bf 76}, Part 2, AMS, Providence, RI, 2007.

\bibitem[GST96]{GST96} F.\ Gesztesy, B.\ Simon and G.\ Teschl,
{\em Zeros of the Wronskian and renormalized oscillation theory}, Amer. J. Math. {\bf 118} (1996),  571--594.
 
 \bibitem[GT09]{GT09} F. Gesztesy and V. Tkachenko, {\em  A criterion for Hill operators to be spectral operators of scalar type,} J. Anal. Math. {\bf 107} (2009), 287--353.
 
\bibitem[GW96]{GW96} F. Gesztesy and R. Weikard, {\em Picard potentials and Hill's equation on a torus,} Acta Math. {\bf 176} (1996), 73--107.

  
 \bibitem[J88]{J88}  C.\ K.\ R.\ T.\ Jones,
{\em Instability of standing waves for nonlinear Schr\"odinger-type equations,} 
Ergodic Theory Dynam. Systems {\bf 8} (1988), 119--138.

\bibitem[JS99]{JS99} D.\ W.\ Jordan and P.\ Smith, {\it Nonlinear Ordinary Differential Equations: An Introduction to Dynamical Systems,} Oxford App. and Engin. Math., Oxford, 1999.

\bibitem[Kr97]{K97} Y. Karpeshina, 
{\em Perturbation Theory for the Schr\"odinger Operator with a Periodic Potential,} 
Lect. Notes Math. {\bf 1663},  Springer-Verlag, Berlin, 1997.

\bibitem[K80]{Kato} T.\ Kato, {\it Perturbation Theory for Linear Operators}, Springer, Berlin, 1980.

\bibitem[MW]{MW} W.\ Magnus and S.\ Winkler, {\em Hill's Equation},  Dover, New York, 1979.

\bibitem[M63]{M63} J.\ Milnor, {\em Morse Theory},  Annals of Math.\ Stud. {\bf 51}, Princeton Univ. Press, Princeton, N.J., 1963. 

 \bibitem[O90]{O90} V.\ Yu.\ Ovsienko,  {\em Selfadjoint differential operators and curves on a Lagrangian Grassmannian that are subordinate to a loop,}  Math. Notes {\bf 47} (1990),  270--275.

 \bibitem[ReSi78]{RS78} M.\ Reed and B.\ Simon, {\it Methods of Modern Mathematical
Physics. IV: Analysis of Operators},  Academic Press, New York, 1978.
	
\bibitem[RS93]{rs93} J.\ Robbin and D.\ Salamon,
{\em The Maslov index for paths}, Topology {\bf 32} (1993), 827--844.

\bibitem[RS95]{RoSa95} J.\ Robbin and D.\ Salamon, {\it  The spectral flow and the Maslov
index},  Bull. London Math. Soc. {\bf 27} (1995),  1--33.

\bibitem[SS08]{SS08} B.\ Sandstede and A.\ Scheel, \emph{Relative Morse indices, Fredholm indices, and group velocities},  Discrete Contin. Dyn. Syst.  \textbf{20}  (2008),  139--158.

\bi[S-B12]{SB} H.\ Schulz-Baldes, 
{\em Sturm intersection theory for periodic Jacobi matrices and linear Hamiltonian systems.} 
Linear Algebra Appl. {\bf 436} (2012), 498--515. 

\end{thebibliography}

\end{document}